\documentclass[11pt]{amsart}
\usepackage[a4paper,margin=29mm]{geometry}
\usepackage[T1]{fontenc}
\usepackage{lmodern,microtype,amsmath,amssymb,mathtools,enumitem}
\usepackage[hidelinks]{hyperref}

\newtheorem{theorem}{Theorem}[section]
\newtheorem{proposition}[theorem]{Proposition}
\newtheorem{lemma}[theorem]{Lemma}

\numberwithin{equation}{section}
\setlist[enumerate]{leftmargin=2em,itemsep=3pt}
\title[Positive Toeplitz operators between mixed-norm Fock spaces]{Positive Toeplitz Operators Between Gaussian Mixed-Norm Fock Spaces}
\author{Deepak K. Dileep}
\address{Department of Applied Mathematics, National Yang Ming Chiao Tung University, Hsinchu 30010, Taiwan}
\email{deepakkd@nycu.edu.tw}
\author{Xiang Fang}
\address{Department of Applied Mathematics, National Yang Ming Chiao Tung University, Hsinchu 30010, Taiwan}
\email{xfang@nycu.edu.tw}
\author{Pham Trong Tien}
\address{VNU University of Science, Vietnam National University, Hanoi, Hanoi, Vietnam}
\email{phamtien@vnu.edu.vn}
\date{}
\begin{document}

\begin{abstract}
We study positive measure-symbol Toeplitz operators between Gaussian
radial--angular mixed-norm Fock spaces throughout the full range of exponents, including infinity. We obtain a complete
characterization of boundedness in terms of weighted local masses of
the symbol measure, or equivalently in terms of its Gaussian Berezin
profile. The criterion separates the transfer of angular integrability
from that of radial summability and remains valid in the quasi-Banach
and endpoint regimes.

\smallskip
\noindent
We further characterize compactness and obtain the corresponding essential-norm estimates through the
asymptotic behavior of the same local-mass and Berezin profiles, and
determine the mappings involving the large and little
outer-endpoint spaces. We also prove an independent characterization
of bounded positive Toeplitz sesquilinear forms. The proofs combine
atomic recovery, localized Gaussian matrix estimates, and Rademacher
extraction.
\end{abstract}

\maketitle

\section{Introduction and main results}\label{sec:intro}

Toeplitz operators form one of the basic classes of operators on spaces of analytic functions, and their boundedness, compactness, spectral properties, and ideal membership have been studied extensively on Hardy, Bergman, and related spaces; see, for example, \cite{ZhuBergmanToeplitz,ZhuOperator} and the references therein. In the Gaussian Fock setting, positive Toeplitz operators with measure symbols admit particularly geometric descriptions in terms of local masses, Fock--Carleson measures, and Berezin transforms; see \cite{IsralowitzZhu,SchusterVarolin,Zhu}. Off-diagonal mappings between classical Fock spaces were studied in \cite{HuLv11,LuLv,Mengestie}, while related questions for generalized, doubling, and large Fock spaces, as well as essential norms, have been considered in \cite{Arroussi,HuLv14,HuLu,HuLv17,Lv,WangXu}.

\smallskip
\noindent
Mixed-norm Fock spaces provide a two-parameter refinement of the classical Fock scale by separating angular integrability on circles from radial summability. Several aspects of this scale have been developed in recent years. Random analytic functions in Fock-type mixed norms were studied in \cite{FangTienRandom}, while Constantin and Peláez established a Littlewood--Paley formula closely related to the radial mixed-norm structure in \cite{ConstantinPelaez}. Blasco and Galbis studied Hausdorff operators on Fock spaces and obtained several results in the mixed-norm setting in \cite{BlascoGalbis}; this direction was subsequently extended by Liu to weighted mixed-norm Fock spaces in \cite{LiuHausdorff}. Liu also studied Fock projections on mixed-norm spaces in \cite{LiuProjection}, and the off-diagonal projection problem was subsequently considered in \cite{LiuShiWei}. More recently, the intrinsic structure of the Gaussian radial--angular scale, including sharp growth, embeddings, atomic decompositions, Carleson measures, and duality, was developed in the companion paper \cite{FT}, while its sampling and interpolation theory was established in \cite{SI}.

\smallskip
\noindent
The aim of the present paper is to develop the positive Toeplitz theory on this radial--angular scale. The main new feature is the interaction between the local distribution of the symbol measure and the two independent integrability structures of the source and target spaces: angular and radial transfer must therefore be treated separately. This distinguishes the problem from the classical one-index Fock theory, where only one integrability parameter is present, and from the mixed-norm projection problem, where no symbol measure enters.

\subsection{The mixed-norm scale and the Toeplitz problem}

Let $\mathcal H(\mathbb C)$ denote the space of entire functions. For
$0<p \leq \infty$ and $r\geq0$, set
\[
 M_p(f,r)
 := \begin{cases}
 \displaystyle
 \left(
 \frac1{2\pi}\int_0^{2\pi}|f(re^{i\theta})|^p\,d\theta
 \right)^{\frac{1}{p}}, \qquad & 0 < p < \infty, \\[3mm]
 \displaystyle
 \max_{|z|=r}|f(z)|, \qquad & p = \infty.
 \end{cases}
\]
For $\alpha>0$ and $0<p,q\leq\infty$, define
\[
 \|f\|_{\mathcal F_\alpha^{p,q}}
 :=
 \begin{cases}
 \displaystyle
 \left(
 \alpha q\int_0^\infty
 M_p(f,r)^q e^{- \frac{\alpha q}{2} r^2}r\,dr
 \right)^{\frac{1}{q}},
 & 0<q<\infty,\\[3mm]
 \displaystyle
 \sup_{r\geq0}M_p(f,r)e^{-\frac{\alpha}{2} r^2},
 & q=\infty.
 \end{cases}
\]
The Gaussian radial--angular mixed-norm Fock space is
\[
 \mathcal F_\alpha^{p,q}
 :=
 \bigl\{
 f\in\mathcal H(\mathbb C):
 \|f\|_{\mathcal F_\alpha^{p,q}}<\infty
 \bigr\}.
\]
When $p=q$, we write
\[
 \mathcal F_\alpha^p=\mathcal F_\alpha^{p,p}.
\]
For the outer endpoint, we also use the little space
\[
 f_\alpha^{p,\infty}
 :=
 \left\{
 f\in\mathcal F_\alpha^{p,\infty}:
 M_p(f,r)e^{-\frac{\alpha}{2} r^2}\longrightarrow0
 \ \text{as } r\to\infty
 \right\}.
\]
For $\sigma\in\mathbb R$, we also write
$\mathcal F_{\alpha;\sigma}^{p,q}$ for the polynomially weighted
mixed-norm Fock space obtained by replacing $M_p(f,r)$ in the above
definition by
\[
 (1+r)^\sigma M_p(f,r).
\]
Thus $\mathcal F_{\alpha;0}^{p,q}=\mathcal F_\alpha^{p,q}$.

\smallskip
\noindent
Throughout the paper, we use the convention
\(
 \frac1\infty=0,
\)
and, for $0<p,q\leq\infty$, put
\[
 \nu(p,q):=\frac1q-\frac1p.
\]
For $k\in\mathbb N_0$, let
\[
 A_k:=\{z\in\mathbb C:k\leq |z|<k+1\}.
\]
If $0<p,q\leq\infty$ and $h$ is measurable on $\mathbb C$, define
\[
 \|h\|_{L_{\rm sh}^{p,q}(\mathbb C)}
 :=
 \begin{cases}
 \displaystyle
 \left(
 \sum_{k=0}^\infty
 \|h\|_{L^p(A_k,dA)}^q
 \right)^{\frac{1}{q}},
 & 0<q<\infty,\\[3mm]
 \displaystyle
 \sup_{k\geq0}\|h\|_{L^p(A_k,dA)},
 & q=\infty,
 \end{cases}
\]
where $dA$ denotes Lebesgue area measure on $\mathbb C$ and
\[
 \|h\|_{L^p(A_k,dA)}
 :=
 \begin{cases}
 \displaystyle
 \left(\int_{A_k}|h(z)|^p\,dA(z)\right)^{\frac{1}{p}},
 & 0<p<\infty,\\[3mm]
 \displaystyle
 \operatorname*{ess\,sup}_{z\in A_k}|h(z)|,
 & p=\infty.
 \end{cases}
\]
We denote by $L_{\rm sh}^{p,q}(\mathbb C)$ the space of measurable
functions for which this quasi-norm is finite.

\smallskip
\noindent
For a positive locally finite Borel measure $\mu$ on $\mathbb C$, define
\[
 T_\mu^\alpha f(z)
 :=
 \frac{\alpha}{\pi}
 \int_{\mathbb C}
 f(w)e^{\alpha z\overline w-\alpha|w|^2}\,d\mu(w).
\]
The paper concerns the matching-Gaussian setting, in which the source
space, the Toeplitz kernel, and the target space have the same Gaussian
parameter $\alpha$.

\smallskip
\noindent
For $R>0$, set
\[
 \widehat\mu_R(z):=\mu(B(z,R)),
 \qquad
 \widetilde\mu_\alpha(z)
 :=
 \int_{\mathbb C}e^{-\alpha|z-w|^2}\,d\mu(w).
\]
We refer to $\widehat\mu_R$ as the local-mass function and to
$\widetilde\mu_\alpha$ as the Gaussian Berezin profile.

\subsection{The Toeplitz operators}

For source and target pairs $(p_1,q_1)$ and $(p_2,q_2)$, write
\[
 \nu_j:=\nu(p_j,q_j)=\frac1{q_j}-\frac1{p_j},
 \qquad j=1,2,
\]
and define
\[
 \frac1{p_{12}}
 :=
 \left(\frac1{p_2}-\frac1{p_1}\right)_+,
 \qquad
 \frac1{q_{12}}
 :=
 \left(\frac1{q_2}-\frac1{q_1}\right)_+.
\]

\smallskip
\noindent
We separate the two cases
\[
 0<q_1\leq q_2\leq\infty \qquad \text{and} \qquad 0<q_2 < q_1\leq\infty.
\]
In the first range, $q_{12}=\infty$. Thus boundedness is expressed
by uniform control of the shell norms, while compactness requires
these norms to vanish at infinity.

\begin{theorem}[The case $q_1\leq q_2$]\label{thm:q1leq-q2}
Let $\alpha>0$,
\( 0<p_1,p_2\leq\infty\), \( 0<q_1\leq q_2\leq\infty\),
and let $\mu$ be a positive locally finite Borel measure on
$\mathbb C$. Then the following are equivalent:
\begin{enumerate}[label=\textup{(\roman*)}]
\item
The operator
\(
 T_\mu^\alpha:
 \mathcal F_\alpha^{p_1,q_1}
 \longrightarrow
 \mathcal F_\alpha^{p_2,q_2}
\)
is bounded.

\item
For some, equivalently every, $R>0$,
\(
 (1+|z|)^{\nu_2-\nu_1}\widehat\mu_R(z)
 \in L_{\rm sh}^{p_{12},\infty}(\mathbb C).
\)

\item
\(
 (1+|z|)^{\nu_2-\nu_1}\widetilde\mu_\alpha(z)
 \in L_{\rm sh}^{p_{12},\infty}(\mathbb C).
\)
\end{enumerate}
Moreover, for every fixed $R>0$,
\[ \|T_\mu^\alpha\|_{\mathcal F_\alpha^{p_1,q_1}
 \longrightarrow
 \mathcal F_\alpha^{p_2,q_2}}
 \asymp
 \big\|
 (1+|\cdot|)^{\nu_2-\nu_1}\widehat\mu_R
 \big\|_{L_{\rm sh}^{p_{12},\infty}}
 \asymp
 \big\|
 (1+|\cdot|)^{\nu_2-\nu_1}\widetilde\mu_\alpha
 \big\|_{L_{\rm sh}^{p_{12},\infty}}.
\]
Furthermore, for every fixed $R>0$,
\[
 \|T_\mu^\alpha\|_{\rm ess}
 \asymp
 \limsup_{k\to\infty}
 \left\|
 (1+|\cdot|)^{\nu_2-\nu_1}\widehat\mu_R
 \right\|_{L^{p_{12}}(A_k,dA)} \asymp \limsup_{k\to\infty}
 \left\|
 (1+|\cdot|)^{\nu_2-\nu_1}\widetilde\mu_\alpha
 \right\|_{L^{p_{12}}(A_k,dA)}.
\]
In particular,
$T_\mu^\alpha$ is compact if and only if
\[
 \lim_{k \to\infty}
 \left\|
 (1+|\cdot|)^{\nu_2-\nu_1}\widehat\mu_R
 \right\|_{L^{p_{12}}(A_k,dA)}
 =0.
\]
This condition is independent of the choice of $R>0$, and is equivalent to
\[
 \lim_{k \to\infty}\left\|
 (1+|\cdot|)^{\nu_2-\nu_1}\widetilde\mu_\alpha
 \right\|_{L^{p_{12}}(A_k,dA)}
 = 0.
\]
The constants in the norm and essential-norm comparisons depend on $\alpha,p_1,q_1,p_2,q_2$ and the fixed
radius $R$.

\end{theorem}

\smallskip
\noindent
In the second range, $q_{12}<\infty$, and boundedness and compactness
of $T_\mu^\alpha$ are equivalent.

\begin{theorem}[The case \(q_2 < q_1\)]
\label{thm:q2lt-q1}
Let $\alpha>0$, \( 0<p_1,p_2\leq\infty\), \( 0<q_2 < q_1\leq\infty\), and let $\mu$ be a positive locally finite Borel
measure on $\mathbb C$. The following assertions are equivalent:
\begin{enumerate}[label=\textup{(\roman*)}]
\item
The operator
\(
 T_\mu^\alpha:
 \mathcal F_\alpha^{p_1,q_1}
 \longrightarrow
 \mathcal F_\alpha^{p_2,q_2}
\) is bounded.

\item
The operator
\(
 T_\mu^\alpha:
 \mathcal F_\alpha^{p_1,q_1}
 \longrightarrow
 \mathcal F_\alpha^{p_2,q_2}
\)
is compact.

\item
For some, equivalently every, $R>0$,
\(
 (1+|z|)^{\nu_2-\nu_1}\widehat\mu_R(z)
 \in L_{\rm sh}^{p_{12},q_{12}}(\mathbb C).
\)

\item
\(
 (1+|z|)^{\nu_2-\nu_1}\widetilde\mu_\alpha(z)
 \in L_{\rm sh}^{p_{12},q_{12}}(\mathbb C).
\)
\end{enumerate}
Moreover, for every fixed $R>0$,
\[
 \|T_\mu^\alpha\|_{\mathcal F_\alpha^{p_1,q_1}
 \longrightarrow
 \mathcal F_\alpha^{p_2,q_2}}
 \asymp
 \big\|
 (1+|\cdot|)^{\nu_2-\nu_1}\widehat\mu_R
 \big\|_{L_{\rm sh}^{p_{12},q_{12}}}
 \asymp
 \big\|
 (1+|\cdot|)^{\nu_2-\nu_1}\widetilde\mu_\alpha
 \big\|_{L_{\rm sh}^{p_{12},q_{12}}}.
\]
The constants in the norm comparisons depend on
$\alpha,p_1,q_1,p_2,q_2$ and the fixed radius $R$.
\end{theorem}

\subsection{Mappings involving the little outer endpoint}

The preceding results determine boundedness and compactness between the
large mixed-norm Fock spaces. We now consider what changes when one of
the outer-endpoint spaces is replaced by its little counterpart
$f_\alpha^{p,\infty}$.

\smallskip
\noindent
There are two natural questions. First, we compare mappings with little
outer-endpoint source to the corresponding mappings with large source.
Second, we consider mappings whose range is required to lie in the
little outer-endpoint space. The latter exhibits a genuine distinction
between finite and infinite outer source exponents.

\begin{theorem}[Little outer-endpoint source]
\label{thm:little-source}
Let $\alpha>0$, $0<p_1,p_2,q_2\leq\infty$, and let $\mu$ be a positive
locally finite Borel measure on $\mathbb C$. Then the operator
\[
 T_\mu^\alpha:
 f_\alpha^{p_1,\infty}
 \longrightarrow
 \mathcal F_\alpha^{p_2,q_2}
\]
has the same boundedness  and compactness criteria as the corresponding operator
\[
 T_\mu^\alpha:
 \mathcal F_\alpha^{p_1,\infty}
 \longrightarrow
 \mathcal F_\alpha^{p_2,q_2}.
\]
If $q_2=\infty$, the same criteria also characterize the operator
\[
 T_\mu^\alpha:
 f_\alpha^{p_1,\infty}
 \longrightarrow
 f_\alpha^{p_2,\infty}.
\]
Moreover, whenever these operators are bounded, the corresponding
norm estimates in the preceding theorems remain valid.
\end{theorem}

\begin{theorem}[Large-to-little outer endpoint]
\label{thm:large-to-little}
Let $\alpha>0$, $0<p_1,p_2\leq\infty$, and let $\mu$ be a positive
locally finite Borel measure on $\mathbb C$.
Then 
\begin{enumerate}[label=\textup{(\alph*)}]
    \item For \(0< q_1 < \infty\), the operator
\[
 T_\mu^\alpha: \mathcal F_\alpha^{p_1,q_1}
 \longrightarrow
  f_\alpha^{p_2,\infty}
\]
has the same boundedness  and compactness criteria as the corresponding operator
\[
 T_\mu^\alpha:
 \mathcal F_\alpha^{p_1,q_1}
 \longrightarrow
 \mathcal F_\alpha^{p_2,\infty}.
\]
\item For \(q_1 = \infty\),
the following assertions are equivalent:
\begin{enumerate}[label=\textup{(\roman*)}]
\item
The operator \(
 T_\mu^\alpha:
 \mathcal F_\alpha^{p_1,\infty}
 \longrightarrow
 f_\alpha^{p_2,\infty}
\)
is bounded.

\item
The operator \(
 T_\mu^\alpha:
 \mathcal F_\alpha^{p_1,\infty}
 \longrightarrow
 f_\alpha^{p_2,\infty}
\)
is compact.

\item
For some, equivalently every, $R>0$,
\(
 \displaystyle \lim_{k \to \infty}\left\|
 (1+|\cdot|)^{\frac{1}{p_1}-\frac{1}{p_2}}\widehat\mu_R
 \right\|_{L^{p_{12}}(A_k,dA)} = 0.
\)

\item
\(
 \displaystyle
 \lim_{k\to\infty} \left\|
 (1+|\cdot|)^{\frac{1}{p_1}-\frac{1}{p_2}}\widetilde\mu_\alpha
 \right\|_{L^{p_{12}}(A_k,dA)} = 0.
\)
\end{enumerate}
\end{enumerate}
In both cases, whenever the operator is bounded, its norm satisfies
the same estimates as the corresponding mapping into
$\mathcal F_\alpha^{p_2,\infty}$.
\end{theorem}

\subsection{Positive Toeplitz sesquilinear forms}
\label{subsec:forms}

The Toeplitz operator is naturally associated with a sesquilinear form.
Indeed, with the Gaussian pairing
\[
 \langle f,g\rangle_\alpha
 :=
 \frac{\alpha}{\pi}
 \int_{\mathbb C}
 f(z)\overline{g(z)}e^{-\alpha|z|^2}\,dA(z),
\]
the reproducing formula gives, whenever the integrals are justified,
\[
 \langle T_\mu^\alpha f,g\rangle_\alpha
 =
 \frac{\alpha}{\pi}
 \int_{\mathbb C}
 f(w)\overline{g(w)}e^{-\alpha|w|^2}\,d\mu(w).
\]
This motivates the positive Toeplitz sesquilinear form
\[
 T_\mu^\alpha(f,g)
 :=
 \frac{\alpha}{\pi}
 \int_{\mathbb C}
 f(w)\overline{g(w)}e^{-\alpha|w|^2}\,d\mu(w).
\]
When the target exponents are in the Banach range, the form criterion
agrees with the corresponding operator criterion. In the quasi-Banach
range, however, a form estimate does not in general determine the
operator quasi-norm. We therefore treat the boundedness of
$T_\mu^\alpha(f,g)$ as an independent problem.

\smallskip
\noindent
For $0<s\leq\infty$, let $s^*$ denote the generalized conjugate
defined by
\[
 \frac1{s^*}
 :=
 \left(1-\frac1s\right)_+ .
\]
For the form problem, define
\[
 \frac1{\hat p}
 :=
 \left(
 1-\frac1{p_1}-\frac1{p_2^*}
 \right)_+,
 \qquad
 \frac1{\hat q}
 :=
 \left(
 1-\frac1{q_1}-\frac1{q_2^*}
 \right)_+,
\]
and put
\[
 \sigma_2
 :=
 \left(\frac1{p_2}-1\right)_+
 -
 \left(\frac1{q_2}-1\right)_+ .
\]
When $p_2,q_2\geq1$, one has
\[
 \hat{p}=p_{12},
 \qquad
 \hat{q}=q_{12},
 \qquad
 \sigma_2=0.
\]
\begin{theorem}[Positive Toeplitz sesquilinear forms]
\label{thm:form}
Let $\alpha>0$, $0<p_1,q_1,p_2,q_2\leq\infty$, and let $\mu$ be a
positive locally finite Borel measure on $\mathbb C$. 
Then the estimate
\[
 |T_\mu^\alpha(f,g)|
\le C_\mu
 \|f\|_{\mathcal F_\alpha^{p_1,q_1}}
 \|g\|_{\mathcal F_{\alpha;\sigma_2}^{p_2^*,q_2^*}}
\]
holds for all
\( f\in\mathcal F_\alpha^{p_1,q_1}\) and \(
 g\in\mathcal F_{\alpha;\sigma_2}^{p_2^*,q_2^*}\),
with some constant $C_\mu>0$,
if and only if, for some, equivalently every, $R>0$,
\[
 (1+|z|)^{\nu_2-\nu_1}\widehat\mu_R(z)
 \in
 L_{\rm sh}^{\hat{p},\hat{q}}(\mathbb C).
\]
Equivalently,
\[
 (1+|z|)^{\nu_2-\nu_1}\widetilde\mu_\alpha(z)
 \in
 L_{\rm sh}^{\hat{p},\hat{q}}(\mathbb C).
\]
Moreover, for every fixed $R>0$,
\[
\|T_\mu^\alpha\|_{\mathrm{form}}
 \asymp
 \big\|
 (1+|\cdot|)^{\nu_2-\nu_1}\widehat\mu_R
 \big\|_{L_{\rm sh}^{\hat{p},\hat{q}}}
 \asymp
 \big\|
 (1+|\cdot|)^{\nu_2-\nu_1}\widetilde\mu_\alpha
 \big\|_{L_{\rm sh}^{\hat{p},\hat{q}}},
\]
where
\[
\|T_\mu^\alpha\|_{\mathrm{form}}
 :=
 \sup_{f\ne0,\,g\ne0}
 \frac{|T_\mu^\alpha(f,g)|}
 {\|f\|_{\mathcal F_\alpha^{p_1,q_1}}
  \|g\|_{\mathcal F_{\alpha;\sigma_2}^{p_2^*,q_2^*}}}.
\]
The constants in the norm comparisons depend on
$\alpha,p_1,q_1,p_2,q_2$ and the fixed radius $R$.
\end{theorem}

\subsection{Proof strategy and organization}
\label{subsec:strategy}

The proofs are organized around two complementary discretizations.
For sufficiency, the pocket decomposition of Subsection~\ref{subsec:pocket} provides
a bounded linear recovery of arbitrary functions in the source space.
After inserting this representation into the Toeplitz integral, the
operator is dominated by Gaussian matrices surrounding a diagonal
multiplier whose coefficients are weighted local masses of the
symbol. The annular multiplier theorem then separates the transfer
of the inner and outer exponents. For necessity, we use instead the
fixed sampling lattice of Subsection~\ref{subsec:lattice}. The associated Toeplitz
testing matrix has diagonal
\[
 \frac{\alpha}{\pi}
 (1+|\lambda|)^{\nu_2-\nu_1}
 \widetilde\mu_\alpha(\lambda),
\]
and finite Rademacher extraction recovers this diagonal without
appealing to duality of the target space. The local-mass
discretization and the Gaussian Berezin comparison then convert the
resulting sequence conditions into the continuous criteria stated in
Theorems~\ref{thm:q1leq-q2} and~\ref{thm:q2lt-q1}.

\smallskip
\noindent
Compactness is obtained from the same framework. Compactly supported
symbols are approximated in operator quasi-norm by finite-rank
operators, while truncation of the measure gives the upper
essential-norm estimate. In the nonvanishing-tail case, sparse
single-shell tests and the off-diagonal decay of the testing matrix
give the reverse estimate. Mappings involving the little outer
endpoint are handled through the canonical dual reconstruction and
the corresponding vanishing-shell sequence spaces. The positive
Toeplitz sesquilinear form is treated separately: its sufficiency
uses a weighted one-Carleson estimate, whereas necessity follows from
finite signed kernel tests, finite sequence factorization, and
finite-support norming. In particular, the form theorem is not used
in the proof of the operator classifications.

\smallskip
\noindent
Section~\ref{sec:tools} collects the analytic inputs from the
mixed-Fock and sampling theories. Section~\ref{sec:discrete}
develops the Gaussian matrix estimates, annular multipliers,
Rademacher extraction, and the square Toeplitz testing matrix.
Section~\ref{sec:comparison} establishes the comparison between local
masses and Gaussian Berezin profiles, including their asymptotic
tails. Section~\ref{sec:operators} proves the operator results,
treats the large and little outer endpoints, and records several particular cases.
Finally, Section~\ref{sec:forms} proves the positive Toeplitz
sesquilinear-form theorem.

\bigskip
\noindent
\textit{Notation.}
Throughout, $C>0$ denotes a constant whose value may change
from line to line. For nonnegative quantities $A$ and $B$,
we write $A\lesssim B$ if $A\le CB$, $A\gtrsim B$ if
$B\lesssim A$, and $A\asymp B$ if both comparisons hold.
Unless otherwise indicated, implicit constants may depend
on the fixed Gaussian parameters, exponents, radii,
localization orders, and geometric discretizations under
consideration, but are independent of the varying functions,
sequences, points, measures, finite subsets, selected
subsequences, and shell or truncation indices. Dependence
on fixed auxiliary parameters may be displayed by subscripts
or named constants when informative. Any uniformity needed
for limiting, compactness, essential-norm, or absorption
arguments is stated explicitly.
In a quasi-Banach target, $\|A\|_{X\to Y}$ denotes the
operator quasi-norm.

\section{Analytic tools from the mixed-Fock and sampling theories}\label{sec:tools}

We collect the companion-paper inputs in the normalization used throughout
this paper. The coefficient spaces are annular; no translation invariance of
their quasi-norms is assumed. Two discretizations will be used repeatedly:
the pocket decomposition provides constructive atomic recovery, while the
fixed Hilbert sampling lattice provides restriction, reconstruction, and the
testing framework for necessity.

\subsection{Spaces, annuli, and kernels}
\label{subsec:mixed-norm-tools}
We retain the notation introduced in Section~\ref{sec:intro}. For
\[
 s(p,q):=\min\{1,p,q\},
\]
the space $\mathcal F_\alpha^{p,q}$ is complete and satisfies
\[
 \|f+g\|_{\mathcal F_\alpha^{p,q}}^{s(p,q)}
 \leq
 \|f\|_{\mathcal F_\alpha^{p,q}}^{s(p,q)}
 +
 \|g\|_{\mathcal F_\alpha^{p,q}}^{s(p,q)};
\]
see \cite[Lemma 2.13]{FT}. The little outer-endpoint space
$f_\alpha^{p,\infty}$ is closed in $\mathcal F_\alpha^{p,\infty}$.

\smallskip
\noindent
For a locally finite configuration $\Lambda\subset\mathbb C$, put
$\Lambda_k:=\Lambda\cap A_k$ and
\[
 \|a\|_{\ell^{p,q}(\Lambda)}
 :=
 \big\|
   (\|a|_{\Lambda_k}\|_{\ell^p})_{k\geq0}
 \big\|_{\ell^q}.
\]
At the outer endpoint, let
\[
 \ell_0^{p,\infty}(\Lambda)
 :=
 \left\{
 a\in\ell^{p,\infty}(\Lambda):
 \|a|_{\Lambda_k}\|_{\ell^p}\to0
 \right\}.
\]
For $\sigma\in\mathbb R$, define
\[
 \|a\|_{\ell_{\mathrm{nat};\sigma}^{p,q}(\Lambda)}
 :=
 \big\|
 \big(
 (1+k)^{\nu(p,q)+\sigma}
 \|a|_{\Lambda_k}\|_{\ell^p}
 \big)_{k\geq0}
 \big\|_{\ell^q},
\]
and, when $q=\infty$,
\[
 \ell_{\mathrm{nat};\sigma,0}^{p,\infty}(\Lambda)
 :=
 \left\{
 a\in\ell_{\mathrm{nat};\sigma}^{p,\infty}(\Lambda):
 (1+k)^{\nu(p,\infty)+\sigma}
 \|a|_{\Lambda_k}\|_{\ell^p}\to0
 \right\}.
\]
We write $\ell_{\mathrm{nat}}^{p,q}$ and
$\ell_{\mathrm{nat},0}^{p,\infty}$ when $\sigma=0$.
These sequence spaces satisfy the corresponding powered
quasi-triangle inequalities; see \cite[Section 2.1]{SI}.

\smallskip
\noindent
The normalized kernel is
\[
 \kappa_{\alpha,w}(z)
 =
 e^{\alpha z\overline w-\frac{\alpha}{2}|w|^2},
 \qquad
 |\kappa_{\alpha,w}(z)|e^{-\frac{\alpha}{2}|z|^2}
 =
 e^{-\frac{\alpha}{2}|z-w|^2}.
\]
We use \cite[Lemma 2.4 and Proposition 2.11]{FT} in the form
\[
 |f(z)|e^{-\frac{\alpha}{2}|z|^2}
 \lesssim
 (1+|z|)^{-\nu(p,q)}
 \|f\|_{\mathcal F_\alpha^{p,q}},
 \qquad
 \|\kappa_{\alpha,w}\|_{\mathcal F_\alpha^{p,q}}
 \asymp
 (1+|w|)^{\nu(p,q)}.
\]

\smallskip
\noindent
For
\[
 B_{p,\sigma}(f;k)
 :=
 \sup_{k\leq r<k+1}
 M_p(f,r)e^{-\frac{\alpha}{2} r^2}(1+r)^\sigma,
\]
annular discretization gives
\begin{equation}\label{eq:annular}
 \|f\|_{\mathcal F_{\alpha;\sigma}^{p,q}}
 \asymp
 \left\|
 \bigl(
 (1+k)^{\frac{1}{q}}B_{p,\sigma}(f;k)
 \bigr)_{k\geq0}
 \right\|_{\ell^q}.
\end{equation}
This includes $q=\infty$ and follows from
\cite[Proposition 2.6 and Lemma 7.1]{FT}.  We write
\(
 B_p(f;k):=B_{p,0}(f;k)
\)
when $\sigma=0$.

\subsection{Atomic recovery from an exact pocket decomposition} \label{subsec:pocket}

We first record the atomic discretization used in the constructive
estimates for the Toeplitz operator.
Fix $0<p,q\leq\infty$. Choose a sufficiently fine pocket tiling
$\mathcal P_\delta$, with centers $\zeta_Q$, as in
\cite[Section 5.1]{FT}, and assign the boundaries by a fixed Borel
convention so that
\[
 \mathbb C
 =
 \bigsqcup_{Q\in\mathcal P_\delta}Q,
 \qquad
 B(\zeta_Q,c_{\rm in}\delta)
 \subset Q
 \subset
 B(\zeta_Q,C_{\rm out}\delta).
\]
The centers form a separated configuration. For $k\in\mathbb N_0$, put
\[
 \mathcal P_{\delta,k}
 :=
 \{Q\in\mathcal P_\delta:\zeta_Q\in A_k\},
\]
and equip sequences on $\mathcal P_\delta$ with the corresponding
annular mixed quasi-norm $\ell^{p,q}(\mathcal P_\delta)$. The above
disjoint union is an exact Borel partition, so it may be used for
arbitrary Borel measures, including measures charging cell boundaries.

\smallskip
\noindent
For $Q\in\mathcal P_\delta$, set
\[
 \widetilde\kappa_{\alpha,\zeta_Q}
 :=
 \frac{\kappa_{\alpha,\zeta_Q}}
 {\|\kappa_{\alpha,\zeta_Q}\|_{\mathcal F_\alpha^{p,q}}}.
\]
The mixed-normalized atomic decomposition
\cite[Corollary 5.33]{FT}, together with the bounded linear recovery
operator recorded in \cite[Section 6.1, equation (6.1)]{FT}, gives
bounded linear maps
\[
 S_\delta^{p,q}:
 \ell^{p,q}(\mathcal P_\delta)
 \longrightarrow
 \mathcal F_\alpha^{p,q},
 \qquad
 A_\delta^{p,q}:
 \mathcal F_\alpha^{p,q}
 \longrightarrow
 \ell^{p,q}(\mathcal P_\delta),
 \qquad
 S_\delta^{p,q}A_\delta^{p,q}=I,
\]
where
\[
 S_\delta^{p,q}d
 =
 \sum_{Q\in\mathcal P_\delta}
 d_Q\widetilde\kappa_{\alpha,\zeta_Q},
\]
with locally uniform convergence. Thus every
$f\in\mathcal F_\alpha^{p,q}$ admits the linear recovery
\[
 f
 =
 S_\delta^{p,q}A_\delta^{p,q}f,
 \qquad
 \|A_\delta^{p,q}f\|_{\ell^{p,q}}
 \lesssim
 \|f\|_{\mathcal F_\alpha^{p,q}}.
\]

\subsection{Localized synthesis and a fixed Hilbert sampling lattice}
\label{subsec:lattice}
We next record a localized synthesis principle and then specialize it to a
fixed Hilbert sampling lattice. The latter will be used for the testing and
necessity arguments.
Let $\Sigma\subset\mathbb C$ be separated. A family
$(\phi_\xi)_{\xi\in\Sigma}$ is rapidly Fock localized if, for every
$N>0$,
\[
 |\phi_\xi(z)|e^{-\frac{\alpha}{2}|z|^2}
 \lesssim_N
 (1+|z-\xi|)^{-N},
 \qquad z\in\mathbb C,\ \xi\in\Sigma.
\]
The implicit constant may depend on $N$ and on the fixed
family and its geometric data.
By \cite[Theorem 2.5]{SI}, for $0<p,q\leq\infty$ and
$\sigma\in\mathbb R$, the synthesis operator
\[
 S_\Sigma^\phi:
 \ell_{\mathrm{nat};\sigma}^{p,q}(\Sigma)
 \longrightarrow
 \mathcal F_{\alpha;\sigma}^{p,q},
 \qquad
 S_\Sigma^\phi c
 :=
 \sum_{\xi\in\Sigma}c_\xi\phi_\xi,
\]
is bounded, and the defining series converges absolutely and locally
uniformly.

\smallskip
\noindent
In particular, \cite[Corollary 2.6]{SI} shows that the normalized synthesis operator
\[
 S_\Sigma^{\phi,p,q}a
 :=
 \sum_{\xi\in\Sigma}
 a_\xi(1+|\xi|)^{-\nu(p,q)}\phi_\xi
\]
is bounded from \(\ell^{p,q}(\Sigma)\) to \(\mathcal F_\alpha^{p,q}\).
When $q=\infty$, it also defines a bounded map
\[
 S_\Sigma^{\phi,p,\infty}:
 \ell_0^{p,\infty}(\Sigma)
 \longrightarrow
 f_\alpha^{p,\infty}.
\]

\smallskip
\noindent
We now fix
\[
 \Omega:=\delta_\Omega(\mathbb Z+i\mathbb Z),
 \qquad
 0<\delta_\Omega<\sqrt{\frac{\pi}{\alpha}}.
\]
Then $\Omega$ is a separated sampling set for
$\mathcal F_\alpha^2$ by \cite[Theorem 2.1]{SI}. Let
$(\varphi_\omega)_{\omega\in\Omega}$ denote its canonical Hilbert dual
family. By \cite[Proposition 2.2]{SI}, these atoms are rapidly Fock
localized.

\smallskip
\noindent
For $0<p,q\leq\infty$, define the normalized restriction
\[
 R_{\alpha,\Omega}^{p,q}f
 :=
 \left(
 f(\omega)e^{-\frac{\alpha}{2}|\omega|^2}
 (1+|\omega|)^{\nu(p,q)}
 \right)_{\omega\in\Omega}.
\]
By \cite[Proposition 3.1]{SI},
\[
 R_{\alpha,\Omega}^{p,q}:
 \mathcal F_\alpha^{p,q}
 \longrightarrow
 \ell^{p,q}(\Omega)
\]
is bounded, and at the outer endpoint
\[
 R_{\alpha,\Omega}^{p,\infty}:
 f_\alpha^{p,\infty}
 \longrightarrow
 \ell_0^{p,\infty}(\Omega)
\]
is bounded.

\smallskip
\noindent
The canonical dual family reconstructs the whole mixed-norm scale:
by \cite[Proposition 3.6]{SI},
\[
 f
 =
 S_\Omega^{\varphi,p,q}R_{\alpha,\Omega}^{p,q}f
 =
 \sum_{\omega\in\Omega}
 f(\omega)e^{-\frac{\alpha}{2}|\omega|^2}\varphi_\omega,
 \qquad
 f\in\mathcal F_\alpha^{p,q}.
\]
For $q<\infty$, the finite-shell partial sums converge in
$\mathcal F_\alpha^{p,q}$. When $q=\infty$, the series converges
absolutely and locally uniformly and still reconstructs every
$f\in\mathcal F_\alpha^{p,\infty}$ by \cite[Proposition 3.6]{SI};
no norm-density argument at the large outer endpoint is used.

\smallskip
\noindent
For later use, taking $\phi_\omega=\kappa_{\alpha,\omega}$ in the generic localized-synthesis result, we write
\[
 S_{\alpha,\Omega}^{p,q}a
 :=
 \sum_{\omega\in\Omega}
 a_\omega(1+|\omega|)^{-\nu(p,q)}
 \kappa_{\alpha,\omega}
\]
for the corresponding normalized-kernel synthesis operator.

\smallskip
\noindent
Finally, let
\[
 Q_\omega
 :=
 \omega+\left[-\frac{\delta_\Omega}{2},\frac{\delta_\Omega}{2}\right) \times \left[-\frac{\delta_\Omega}{2},\frac{\delta_\Omega}{2}\right),
 \qquad \omega\in\Omega.
\]
Then
\[
 \mathbb C=\bigsqcup_{\omega\in\Omega}Q_\omega
\]
is an exact Borel partition, and
\[
 Q_\omega\subset B(\omega,\Delta_\Omega),
 \qquad
 \Delta_\Omega:=\frac{\delta_{\Omega}}{\sqrt2}.
\]
With the pocket parameter $\delta$ fixed in Subsection~2.2, choose once
and for all a working radius
\begin{equation}\label{eq-R*}
 R_*
 >
 \max\{\Delta_\Omega,C_{\rm out}\delta\}.
\end{equation}

\section{Discrete and localized matrix tools}\label{sec:discrete}

The analytic discretizations of Section~\ref{sec:tools} reduce the main operator
estimates to a small collection of sequence-space mechanisms. We develop
them here: Gaussian matrices control off-diagonal transport, diagonal
multipliers encode the change of source and target exponents, and
Rademacher extraction recovers the diagonal information needed for
necessity. These ingredients are then assembled in the square Toeplitz
testing matrix.

\smallskip
\noindent
The configurations below are separated and hence have bounded geometry:
the number of points in any unit disk is uniformly bounded.

\subsection{Gaussian matrices}

We first record the off-diagonal decay of Gaussian matrices on the
annular mixed sequence spaces.

\begin{proposition}[Gaussian matrices on annular mixed sequence spaces]
\label{prop:gaussian-matrices}
Let $\Lambda,\Sigma\subset\mathbb C$ be separated and let
\[
 G_{\xi,\lambda}=e^{-c|\xi-\lambda|^2},
 \qquad c>0.
\]
Then, for every $0<p\leq\infty$,
\begin{equation}\label{eq:gaussian-shell}
 \big\|
 \mathbf 1_{\Sigma_\ell}G\mathbf 1_{\Lambda_k}x
 \big\|_{\ell^p}
 \lesssim
 e^{-\frac c2|\ell-k|^2}
 \big\|x|_{\Lambda_k}\big\|_{\ell^p},
 \qquad k,\ell\in\mathbb N_0.
\end{equation}
The implicit constant depends only on $c,p$ and the separation
constants of $\Lambda$ and $\Sigma$.
Consequently,
\(
 G:\ell^{p,q}(\Lambda)\longrightarrow\ell^{p,q}(\Sigma)
\)
is bounded for all $0<p,q\leq\infty$, and
\(
 G:\ell_0^{p,\infty}(\Lambda)
 \longrightarrow
 \ell_0^{p,\infty}(\Sigma)
\)
is bounded.
\end{proposition}
\begin{proof}
If $\xi\in\Sigma_\ell$ and $\lambda\in\Lambda_k$, then
\[
 |\xi-\lambda|
 \geq
 \big||\xi|-|\lambda|\big|
 \geq
 (|\ell-k|-1)_+.
\]
Then
\[
 e^{-c|\xi-\lambda|^2}
 \lesssim_c
 e^{-\frac c2|\ell-k|^2}
 e^{-\frac c4|\xi-\lambda|^2}.
\]
Set
\[
 H_{\xi,\lambda}
 :=
 e^{-\frac c4|\xi-\lambda|^2}.
\]
Since $\Lambda$ and $\Sigma$ are separated, the Gaussian decay gives
\[
 \sup_{\xi\in\Sigma}\sum_{\lambda\in\Lambda}H_{\xi,\lambda}
 +
 \sup_{\lambda\in\Lambda}\sum_{\xi\in\Sigma}H_{\xi,\lambda}
 <\infty.
\]
For $1\leq p\leq\infty$, the Schur test therefore gives
\[
 \|Hx\|_{\ell^p(\Sigma)}
 \lesssim
 \|x\|_{\ell^p(\Lambda)}.
\]
Hence
\[
 \big\|
 \mathbf 1_{\Sigma_\ell}G\mathbf 1_{\Lambda_k}x
 \big\|_{\ell^p}
 \lesssim
 e^{-\frac c2|\ell-k|^2}
 \big\|x|_{\Lambda_k}\big\|_{\ell^p}.
\]
Suppose now that $0<p<1$. By $p$-subadditivity,
\[
 \sum_{\xi\in\Sigma_\ell}
 \left|
 \sum_{\lambda\in\Lambda_k}
 H_{\xi,\lambda}x_\lambda
 \right|^p
\leq
 \sum_{\xi\in\Sigma_\ell}
 \sum_{\lambda\in\Lambda_k}
 H_{\xi,\lambda}^p|x_\lambda|^p  =
 \sum_{\lambda\in\Lambda_k}
 |x_\lambda|^p
 \sum_{\xi\in\Sigma_\ell}H_{\xi,\lambda}^p.
\]
The last inner sum is uniformly bounded, again by separation. This
proves \eqref{eq:gaussian-shell} for all $0<p\leq\infty$.

\smallskip
\noindent
To pass to the mixed sequence spaces, put
\[
 s :=\min\{1,p\}.
\]
By \eqref{eq:gaussian-shell} and the $s$-triangle inequality,
\[
\big\|(Gx)|_{\Sigma_\ell}\big\|_{\ell^p}^{s}
 \lesssim
 \sum_{k\geq0}
 e^{-\frac{cs}{2}|\ell-k|^2}\big\|x|_{\Lambda_k}\big\|_{\ell^p}^{s}.
\]
Thus $\left(\|(Gx)|_{\Sigma_\ell}\big\|_{\ell^p}^{s}\right)_{\ell \geq 0}$ is dominated by the convolution of
$\left(\big\|x|_{\Lambda_k}\big\|_{\ell^p}^{s}\right)_{k \geq 0}$ with an $\ell^1$ Gaussian sequence.

\smallskip
\noindent
If $q<\infty$ and $\frac{q}{s}\geq1$, Young's inequality on
$\ell^{\frac{q}{s}}$ gives
\[
 \left\|
 \left(
   \|(Gx)|_{\Sigma_\ell}\|_{\ell^p}^{s}
 \right)_{\ell\geq0}
 \right\|_{\ell^{\frac{q}{s}}}
 \lesssim
 \left\|
 \left(
   \|x|_{\Lambda_k}\|_{\ell^p}^{s}
 \right)_{k\geq0}
 \right\|_{\ell^{\frac{q}{s}}}.
\]
Equivalently,
\[
 \|Gx\|_{\ell^{p,q}(\Sigma)}
 \lesssim
 \|x\|_{\ell^{p,q}(\Lambda)}.
\]
If $\frac{q}{s}<1$, the subadditivity of the $\frac{q}{s}$ power gives
\[
 \sum_{\ell\geq0}
 \|(Gx)|_{\Sigma_\ell}\|_{\ell^p}^{q}
 \lesssim
 \sum_{\ell\geq0}\sum_{k\geq0}
 e^{-\frac{cq}{2}|\ell-k|^2}
 \|x|_{\Lambda_k}\|_{\ell^p}^{q}  \lesssim
 \sum_{k\geq0}
 \|x|_{\Lambda_k}\|_{\ell^p}^{q},
\]
which gives the same conclusion.
For $q=\infty$, taking
suprema yields
\[
 \sup_{\ell \geq 0} \|(Gx)|_{\Sigma_\ell}\big\|_{\ell^p}
 \lesssim
 \sup_{k \geq 0}  \big\|x|_{\Lambda_k}\big\|_{\ell^p}.
\]
Hence
\(
 G:\ell^{p,q}(\Lambda)\longrightarrow\ell^{p,q}(\Sigma)
\)
is bounded for every $0<p,q\leq\infty$.

\smallskip
\noindent
Finally, if $x\in\ell_0^{p,\infty}(\Lambda)$, then
$\|x|_{\Lambda_k}\|_{\ell^p}^{s}\to0$.
Since convolution with an $\ell^1$ Gaussian sequence preserves
vanishing at infinity, the preceding estimate gives
$\|(Gx)|_{\Sigma_\ell}\|_{\ell^p}^{s}\to0$, and hence
$\|(Gx)|_{\Sigma_\ell}\|_{\ell^p}\to0$.
Therefore
\(
 G:\ell_0^{p,\infty}(\Lambda)
 \longrightarrow
 \ell_0^{p,\infty}(\Sigma)
\)
is bounded.
\end{proof}

\subsection{Annular diagonal multipliers}

The preceding proposition controls the off-diagonal spreading of
coefficients. We now isolate the diagonal mechanism that records the
change between the source and target exponent pairs.
Let $\Lambda\subset\mathbb C$ be separated and let
$m:=(m_\lambda)_{\lambda\in\Lambda}$ be a scalar sequence. We write
\[
 M_m x:=(m_\lambda x_\lambda)_{\lambda\in\Lambda}
\]
for the associated diagonal multiplier.

\begin{proposition}
\label{prop:annular-multipliers}
Let $0<p_1,q_1,p_2,q_2\leq\infty$. Then the operator
\(
 M_m:\ell^{p_1,q_1}(\Lambda)
 \longrightarrow
 \ell^{p_2,q_2}(\Lambda)
\)
is bounded if and only if
\(
 m  \in \ell^{p_{12}, q_{12}}(\Lambda).
\)
Moreover,
\[
 \|M_m\|_{\ell^{p_1,q_1}\to\ell^{p_2,q_2}}
 =
 \left\| m \right\|_{\ell^{p_{12}, q_{12}}}.
\]
If $q_1=\infty$, replacing the source space
$\ell^{p_1,\infty}(\Lambda)$ by
$\ell_0^{p_1,\infty}(\Lambda)$ does not change either the
boundedness criterion or the operator norm.

\smallskip
\noindent
When $q_2=\infty$, the corresponding little-target assertions are as
follows:
\begin{enumerate}[label=\textup{(\roman*)}]
\item if $q_1<\infty$, boundedness into
      $\ell^{p_2,\infty}(\Lambda)$ is equivalent to boundedness into
      $\ell_0^{p_2,\infty}(\Lambda)$;

\item if $q_1=\infty$, then the operator
\(
 M_m:\ell_0^{p_1,\infty}(\Lambda)
 \longrightarrow
 \ell_0^{p_2,\infty}(\Lambda)
\)
is bounded if and only if
\(
 m  \in \ell^{p_{12}, \infty}(\Lambda);
\)

\item the operator
 \(
 M_m:\ell^{p_1,\infty}(\Lambda)
 \longrightarrow
 \ell_0^{p_2,\infty}(\Lambda)
\)
is bounded if and only if
\( m\in\ell_0^{p_{12},\infty}(\Lambda).
\)
\end{enumerate}
In each case, the operator norm equals the inherited
$\ell^{p_{12},\infty}$ quasi-norm of $m$.
\end{proposition}

\begin{proof}
The argument has two levels: we first compute the multiplier norm on
each annular block and then apply the same calculation to the shell
index.

\smallskip
\noindent
For every $k\geq0$,
\begin{equation}\label{eq:block-multiplier}
 \|M_{m|_{\Lambda_k}}\|_{\ell^{p_1}(\Lambda_k)\to
 \ell^{p_2}(\Lambda_k)}
 =
 \|m|_{\Lambda_k}\|_{\ell^{p_{12}}}.
\end{equation}
Indeed, if $p_2\geq p_1$, then $p_{12}=\infty$ and
\[
 \|m x\|_{\ell^{p_2}}
 \leq
 \|m x\|_{\ell^{p_1}}
 \leq
 \|m|_{\Lambda_k}\|_{\ell^\infty}
 \|x\|_{\ell^{p_1}}.
\]
Testing on a coordinate vector gives equality.

\smallskip
\noindent
If $p_2<p_1<\infty$, then
\[
 \frac1{p_2}
 =
 \frac1{p_1}+\frac1{p_{12}}.
\]
After raising to the $p_2$ power, H\"older's inequality with exponents
$\frac{p_1}{p_2}$ and $\frac{p_{12}}{p_2}$ gives
\[
 \|m x\|_{\ell^{p_2}}
 \leq
 \|m|_{\Lambda_k}\|_{\ell^{p_{12}}}
 \|x\|_{\ell^{p_1}}.
\]
If $m|_{\Lambda_k}\neq0$, equality is attained by taking
\[
 x_\lambda
 =
 \left(
 \frac{|m_\lambda|}
 {\|m|_{\Lambda_k}\|_{\ell^{p_{12}}}}
 \right)^{\frac{p_{12}}{p_1}},
 \qquad \lambda\in\Lambda_k.
\]
Finally, if $p_1=\infty$ and $p_2<\infty$, then $p_{12}=p_2$ and
\[
 \|m x\|_{\ell^{p_2}}
 \leq
 \|m|_{\Lambda_k}\|_{\ell^{p_2}}\|x\|_{\ell^\infty},
\]
with equality for a unimodular choice of $x$ on the support of $m$.
This proves \eqref{eq:block-multiplier}.

\medskip
\noindent
Put
\[
 W_k:=\|m|_{\Lambda_k}\|_{\ell^{p_{12}}},
 \qquad k\geq0.
\]
For $x\in\ell^{p_1,q_1}(\Lambda)$,
\eqref{eq:block-multiplier} gives
\[
 \|(M_mx)|_{\Lambda_k}\|_{\ell^{p_2}}
 \leq
 W_k\|x|_{\Lambda_k}\|_{\ell^{p_1}}.
\]
Hence the action on the shell norms is dominated by the scalar
multiplier
\[
 M_W:\ell^{q_1}\longrightarrow\ell^{q_2},
 \qquad
 (t_k)_{k\geq 0}\longmapsto(W_kt_k)_{k\geq 0}.
\]
Applying the one-level multiplier formula again,
\[
 \|M_W\|_{\ell^{q_1}\to\ell^{q_2}}
 =
 \|W\|_{\ell^{q_{12}}}
 =
 \|m\|_{\ell^{p_{12},q_{12}}},
\]
and therefore
\[
 \|M_m\|
 \leq
 \|m\|_{\ell^{p_{12},q_{12}}}.
\]
For the reverse inequality, fix finitely many nonempty shells and,
on each of them, choose a unit vector $u^{(k)}$ attaining
\eqref{eq:block-multiplier}. Given a finitely supported scalar
sequence $(t_k)_{k\geq 0}$ on these shells, define
\[
 x|_{\Lambda_k}:=t_k u^{(k)}.
\]
Then
\[
 \|x\|_{\ell^{p_1,q_1}}
 =
 \|(t_k)_{k\geq 0}\|_{\ell^{q_1}}
\qquad \text{and} \qquad
 \|M_mx\|_{\ell^{p_2,q_2}}
 =
 \|(W_kt_k)_{k\geq 0}\|_{\ell^{q_2}}.
\]
Thus
\[
 \|M_W\|_{\ell^{q_1}\to\ell^{q_2}}
 \leq
 \|M_m\|_{\ell^{p_1,q_1}\to\ell^{p_2,q_2}}
\]
on every finite collection of shells. Increasing the collection and
using monotonicity gives
\[
 \|m\|_{\ell^{p_{12},q_{12}}}
 \leq
 \|M_m\|_{\ell^{p_1,q_1}\to\ell^{p_2,q_2}}.
\]
This proves the first assertion and the exact norm identity.

\smallskip
\noindent
Suppose $q_1=\infty$. Restriction gives the large-to-little-source
upper bound. Conversely, the preceding lower-bound argument uses only
finitely many shells, and hence all its test vectors belong to
$\ell_0^{p_1,\infty}(\Lambda)$. Therefore the little source has the
same criterion and the same operator norm.

\smallskip
\noindent
Now let $q_2=\infty$. If $q_1<\infty$, then for every
$x\in\ell^{p_1,q_1}(\Lambda)$,
\(
 \|x|_{\Lambda_k}\|_{\ell^{p_1}}\to0.
\) 
Since boundedness is equivalent to
\[
 \sup_{k \geq 0} W_k<\infty,
\]
we obtain
\[
 \|(M_mx)|_{\Lambda_k}\|_{\ell^{p_2}}
 \leq
 W_k\|x|_{\Lambda_k}\|_{\ell^{p_1}}
 \longrightarrow0.
\]
Thus the range automatically lies in
$\ell_0^{p_2,\infty}(\Lambda)$.

\smallskip
\noindent
If $q_1=q_2=\infty$, a bounded sequence $(W_k)_{k \geq 0}$ sends vanishing shell
norms to vanishing shell norms. Hence
\(
 M_m:\ell_0^{p_1,\infty}\to\ell_0^{p_2,\infty}
\)
is bounded exactly when $m\in\ell^{p_{12},\infty}$.

\smallskip
\noindent
Finally,
\(
 M_m:\ell^{p_1,\infty}\to\ell_0^{p_2,\infty}
\)
is bounded if $W_k\to0$. Conversely, assume this mapping is bounded.
The preceding large-target criterion gives $\sup_{k \geq 0} W_k<\infty$.
For each nonempty shell choose a unit vector $u^{(k)}$ attaining
\eqref{eq:block-multiplier}, and set
\(
 x|_{\Lambda_k}:=u^{(k)}
\)
on nonempty shells and zero otherwise. Then
$x\in\ell^{p_1,\infty}(\Lambda)$ and
\[
 \|(M_mx)|_{\Lambda_k}\|_{\ell^{p_2}}
 =
 W_k.
\]
Since $M_mx\in\ell_0^{p_2,\infty}(\Lambda)$, we have $W_k\to0$.
Thus
\[
 m\in\ell_0^{p_{12},\infty}(\Lambda).
\]
The same block tests give the stated exact operator norms.
\end{proof}

\subsection{Diagonal extraction}\label{subsec:diagonal}

Let $B$ be a linear operator acting on sequences indexed by $\Omega$.
For $\omega,\xi\in\Omega$, write
\[
 B_{\omega,\xi}:=(Be_\xi)_\omega
\]
for its matrix coefficients and set
\[
 d_\omega:=B_{\omega,\omega},
 \qquad \omega\in\Omega.
\]
Thus $M_d$ is the diagonal multiplier associated with the diagonal of
$B$.

\smallskip
\noindent
The next lemma shows that, throughout the full quasi-Banach range,
the diagonal multiplier is controlled by the operator itself.  More
precisely, on every finite set the diagonal contribution can be
detected by a suitable choice of Rademacher signs.  The finite-set
form will be used later without any density argument at the large
outer endpoint.

\begin{lemma}[Rademacher diagonal extraction]
\label{lem:diagonal-extraction}
Let $0<p_1,q_1,p_2,q_2\leq\infty$, and suppose that
\(
 B:\ell^{p_1,q_1}(\Omega)
 \longrightarrow
 \ell^{p_2,q_2}(\Omega)
\)
is bounded. Then
\[
 \|M_d\|_{\ell^{p_1,q_1}\to\ell^{p_2,q_2}}
 \lesssim
 \|B\|_{\ell^{p_1,q_1}\to\ell^{p_2,q_2}}.
\]
Moreover, if $E\subset\Omega$ is finite and $x$ is supported in $E$,
then there are signs $\varepsilon_\omega\in\{-1,1\}$,
$\omega\in E$, such that
\[
 \|\mathbf 1_E Bx^\varepsilon\|_{\ell^{p_2,q_2}}
 \gtrsim
 \|(d_\omega x_\omega)_{\omega\in E}\|_{\ell^{p_2,q_2}},
 \qquad
 x^\varepsilon_\omega:=\varepsilon_\omega x_\omega.
\]
The finite-set assertion remains valid when $B$ is defined only on
finitely supported sequences. 
The implicit constants depend only on the target exponents
$p_2,q_2$.
\end{lemma}

\begin{proof}
Set
\[
 s_2:=\min\{1,p_2,q_2\}.
\]
We first work with a finitely supported sequence $x$. Let
$(\varepsilon_\omega)_{\omega \in \Omega}$ be independent Rademacher signs and put
\[
 x^\varepsilon_\omega:=\varepsilon_\omega x_\omega.
\]
For each $\omega$ in the support of $x$,
\[
 (Bx^\varepsilon)_\omega
 =
 \varepsilon_\omega d_\omega x_\omega+y_\omega,
\qquad \text{where} \qquad 
 y_\omega
 :=
 \sum_{\xi\neq \omega}
 B_{\omega, \xi}\varepsilon_\xi x_\xi
\]
is independent of $\varepsilon_\omega$.
Since $0<s_2\leq1$,
\[
 2^{s_2}|a|^{s_2}
 \leq
 |a+y|^{s_2}+|-a+y|^{s_2}.
\]
Averaging this inequality over the remaining Rademacher signs and
using the iterated-expectation identity gives
\[
 |d_\omega x_\omega|^{s_2}
 \le
 2^{1-s_2}
 \mathbb E_\varepsilon
 |(Bx^\varepsilon)_\omega|^{s_2}.
\]
Now
\(
 \ell^{\frac{p_2}{s_2},\frac{q_2}{s_2}}(\Omega)
\)
is a Banach lattice, with the usual interpretation when one of the
exponents is infinite, and
\[
 \|z\|_{\ell^{p_2,q_2}}^{s_2}
 =
 \big\|(|z_\omega|^{s_2})_{\omega\in\Omega}
 \big\|_{\ell^{\frac{p_2}{s_2},\frac{q_2}{s_2}}}.
\]
Hence monotonicity followed by Minkowski's inequality gives
\[
 \|M_dx\|_{\ell^{p_2,q_2}}^{s_2}
 \leq
 2^{1-s_2}
 \left\|
 \mathbb E_\varepsilon
 \big(|(Bx^\varepsilon)_\omega|^{s_2}\big)_{\omega\in\Omega}
 \right\|_{\ell^{\frac{p_2}{s_2},\frac{q_2}{s_2}}}  \leq
 2^{1-s_2}
 \mathbb E_\varepsilon
 \|Bx^\varepsilon\|_{\ell^{p_2,q_2}}^{s_2}.
\]
Since changing signs does not change the
$\ell^{p_1,q_1}$ quasi-norm,
\(
 \|x^\varepsilon\|_{\ell^{p_1,q_1}}
 =
 \|x\|_{\ell^{p_1,q_1}},
\)
and therefore
\[
 \|M_dx\|_{\ell^{p_2,q_2}}
 \lesssim
 \|B\|_{\ell^{p_1,q_1}\to\ell^{p_2,q_2}}
 \|x\|_{\ell^{p_1,q_1}}
\]
for every finitely supported $x$.

\smallskip
\noindent
For a general $x\in\ell^{p_1,q_1}(\Omega)$, let
\[
 x^{(N)}
 :=
 x\,\mathbf 1_{\bigcup_{k=0}^N\Omega_k}.
\]
Since $\Omega$ is separated, $x^{(N)}$ has finite support, and
\[
 \|x^{(N)}\|_{\ell^{p_1,q_1}}
 \leq
 \|x\|_{\ell^{p_1,q_1}}.
\]
Applying the preceding estimate to $x^{(N)}$ and then letting
$N\to\infty$, monotonicity of the shell norms gives
\[
 \|M_dx\|_{\ell^{p_2,q_2}}
 \lesssim
 \|B\|_{\ell^{p_1,q_1}\to\ell^{p_2,q_2}}
 \|x\|_{\ell^{p_1,q_1}}.
\]
Thus
\[
 \|M_d\|_{\ell^{p_1,q_1}\to\ell^{p_2,q_2}}
 \lesssim
 \|B\|_{\ell^{p_1,q_1}\to\ell^{p_2,q_2}}.
\]

\smallskip
\noindent
It remains to prove the finite-set assertion. Let $E\subset\Omega$ be
finite and suppose that $x$ is supported in $E$. Restricting the preceding argument to the coordinates in $E$ gives
\[
 \|(d_\omega x_\omega)_{\omega\in E}\|_{\ell^{p_2,q_2}}^{s_2}
 \leq
 2^{1-s_2}
 \mathbb E_\varepsilon
 \|\mathbf 1_E Bx^\varepsilon\|_{\ell^{p_2,q_2}}^{s_2}.
\]
Hence, for at least one choice of the signs,
\[
 \|\mathbf 1_E Bx^\varepsilon\|_{\ell^{p_2,q_2}}
 \gtrsim
 \|(d_\omega x_\omega)_{\omega\in E}\|_{\ell^{p_2,q_2}}.
\]
This argument only involves finitely supported input vectors, so the
same conclusion remains valid when $B$ is defined only on finitely
supported sequences.
\end{proof}

\subsection{Toeplitz matrices on the sampling lattice}
\label{subsec:toeplitz-matrices}

We now combine the preceding three mechanisms in two matrices
associated with the Toeplitz operator on the fixed sampling lattice.
The first is adapted to kernel testing, while the second will be used
for reconstruction at the outer endpoint.
We retain the fixed sampling lattice $\Omega$ and the associated
square partition $(Q_\omega)_{\omega\in\Omega}$ from
Subsection~\ref{subsec:lattice}. Whenever the integral is defined, set
\[
 C^\mu_{\gamma,\lambda}
 :=
 \frac{\alpha}{\pi}
 (1+|\gamma|)^{\nu_2}
 (1+|\lambda|)^{-\nu_1}
 e^{-\frac{\alpha}{2}|\gamma|^2}
 \int_{\mathbb C}
 \kappa_{\alpha,\lambda}(w)
 e^{\alpha\gamma\overline w-\alpha|w|^2}\,d\mu(w),
 \qquad \gamma,\lambda\in\Omega.
\]
On finitely supported inputs, $C^\mu$ is the matrix of
\[
 R_{\alpha,\Omega}^{p_2,q_2}
 T_\mu^\alpha
 S_{\alpha,\Omega}^{p_1,q_1}.
\]
For $\omega\in\Omega$, define the weighted local-mass
sequence $m=(m_\omega)_{\omega\in\Omega}$ by
\begin{equation}\label{eq-weighted-local-mass}
m_\omega
:=
(1+|\omega|)^{\nu_2-\nu_1}
\widehat\mu_{R_*}(\omega),
\qquad \omega\in\Omega,
\end{equation}
where $R_*$ is fixed in \eqref{eq-R*}.

\begin{proposition}[Square Toeplitz testing matrix]
\label{prop:toeplitz-testing-matrix}
Let $C^\mu$ and $m=(m_\omega)_{\omega\in\Omega}$ be as above.

\begin{enumerate}[label=\textup{(\alph*)}]
\item The diagonal of $C^\mu$ is given by
\begin{equation}\label{eq:testing-diagonal}
 C^\mu_{\lambda,\lambda}
 =
 \frac{\alpha}{\pi}
 (1+|\lambda|)^{\nu_2-\nu_1}
 \widetilde\mu_\alpha(\lambda),
 \qquad \lambda\in\Omega.
\end{equation}

\item For all $\gamma,\lambda\in\Omega$,
\[
 |C^\mu_{\gamma,\lambda}|
 \lesssim
 \sum_{\omega\in\Omega}
 e^{-\frac{\alpha}{4}|\gamma-\omega|^2}
 m_\omega
 e^{-\frac{\alpha}{4}|\omega-\lambda|^2}.
\]

\item If \(m \in \ell^{p_{12}, \infty}(\Omega)\),
then
\begin{equation}\label{eq:testing-shell}
 \|\mathbf 1_{\Omega_\ell}C^\mu x\|_{\ell^{p_2}}
 \lesssim
 \|m\|_{\ell^{p_{12}, \infty}} e^{-\frac{\alpha}{16}|\ell-k|^2}
 \|x\|_{\ell^{p_1}},
 \qquad
 \operatorname{supp}x\subset\Omega_k.
\end{equation}
Moreover, let $\varepsilon_0>0$, and suppose that $k_j\to\infty$ and that $x^{(j)}$ is supported in
$\Omega_{k_j}$ with
\[
 \|x^{(j)}\|_{\ell^{p_1}}\leq1,
 \qquad
 \|\mathbf 1_{\Omega_{k_j}}C^\mu x^{(j)}\|_{\ell^{p_2}}
 \geq\varepsilon_0.
\]
Then, after passing to a sufficiently sparse subsequence of the
shells, the vectors $C^\mu x^{(j)}$ are pairwise separated in
$\ell^{p_2,q_2}(\Omega)$.
\end{enumerate}
The implicit constants in the estimates of parts~(b)
and~(c) depend only on $\alpha$, the fixed exponent pairs,
the working radius $R_*$, and the fixed sampling-lattice
geometry. In the final assertion
of part~(c), the spacing of the selected subsequence may also depend on $\varepsilon_0$ and
$\|m\|_{\ell^{p_{12},\infty}(\Omega)}$.
\end{proposition}

\begin{proof}
We prove the three assertions in order.

\smallskip
\noindent
(a) Taking $\gamma=\lambda$ in the definition of $C^\mu$ gives
\[
 e^{-\frac{\alpha}{2}|\lambda|^2}
 \kappa_{\alpha,\lambda}(w)
 e^{\alpha\lambda\overline w-\alpha|w|^2}
 =
 e^{-\alpha|w-\lambda|^2}.
\]
Therefore
\[
C^\mu_{\lambda,\lambda}
 =
 \frac{\alpha}{\pi}
 (1+|\lambda|)^{\nu_2-\nu_1}
 \int_{\mathbb C}
 e^{-\alpha|w-\lambda|^2}\,d\mu(w)=
 \frac{\alpha}{\pi}
 (1+|\lambda|)^{\nu_2-\nu_1}
 \widetilde\mu_\alpha(\lambda).
\]

\smallskip
\noindent
(b) Taking absolute values in the definition of $C^\mu$, we obtain
\[
 |C^\mu_{\gamma,\lambda}|
 \lesssim
 (1+|\gamma|)^{\nu_2}
 (1+|\lambda|)^{-\nu_1}
 \int_{\mathbb C}
 e^{-\frac{\alpha}{2}|w-\gamma|^2}
 e^{-\frac{\alpha}{2}|w-\lambda|^2}\,d\mu(w).
\]
Decompose the integral over the exact square partition
\[
 \mathbb C=\bigsqcup_{\omega\in\Omega}Q_\omega.
\]
Since $Q_\omega\subset B(\omega, \Delta_\Omega)$, for $w\in Q_\omega$
the bounded displacement from $w$ to $\omega$, together with
polynomial moderation, gives
\[
 (1+|\gamma|)^{\nu_2}
 e^{-\frac{\alpha}{2}|w-\gamma|^2}
 \lesssim
 (1+|\omega|)^{\nu_2}
 e^{-\frac{\alpha}{4}|\gamma-\omega|^2},
\]
and similarly
\[
 (1+|\lambda|)^{-\nu_1}
 e^{-\frac{\alpha}{2}|w-\lambda|^2}
 \lesssim
 (1+|\omega|)^{-\nu_1}
 e^{-\frac{\alpha}{4}|\omega-\lambda|^2}.
\]
Hence
\[
  |C^\mu_{\gamma,\lambda}|
 \lesssim
 \sum_{\omega\in\Omega}
 e^{-\frac{\alpha}{4}|\gamma-\omega|^2}
 (1+|\omega|)^{\nu_2-\nu_1}
 \mu(Q_\omega)
 e^{-\frac{\alpha}{4}|\omega-\lambda|^2}\leq
 \sum_{\omega\in\Omega}
 e^{-\frac{\alpha}{4}|\gamma-\omega|^2}
 m_\omega
 e^{-\frac{\alpha}{4}|\omega-\lambda|^2},
\]
because $Q_\omega\subset B(\omega,R_*)$.

\smallskip
\noindent
(c) Let
\(
 G_{\xi,\zeta}
 :=
 e^{-\frac{\alpha}{4}|\xi-\zeta|^2}
\) and \(M: = \|m\|_{\ell^{p_{12}, \infty}}\).
By part~(b),
\[
 |C^\mu x|\lesssim GM_mG|x|.
\]
Suppose that $\operatorname{supp}x\subset\Omega_k$. By
Proposition~\ref{prop:gaussian-matrices},
\[
 \|\mathbf 1_{\Omega_j}G|x|\|_{\ell^{p_1}}
 \lesssim
 e^{-\frac{\alpha}{8}|j-k|^2}
 \|x\|_{\ell^{p_1}}.
\]
By \eqref{eq:block-multiplier} in Proposition~\ref{prop:annular-multipliers},
\[
 \|\mathbf 1_{\Omega_j}M_mG|x|\|_{\ell^{p_2}}
 \leq
 \|m|_{\Omega_j}\|_{\ell^{p_{12}}}
 \|\mathbf 1_{\Omega_j}G|x|\|_{\ell^{p_1}} \lesssim
 M e^{-\frac{\alpha}{8}|j-k|^2}
 \|x\|_{\ell^{p_1}}.
\]
Set
\[
 s:=\min\{1,p_2\}.
\]
Applying Proposition~\ref{prop:gaussian-matrices} to the outer
Gaussian and using the $s$-triangle inequality,
\[
 \|\mathbf 1_{\Omega_\ell}C^\mu x\|_{\ell^{p_2}}^s
 \lesssim
 \sum_{j\geq0}
 e^{-\frac{\alpha s}{8}|\ell-j|^2}
 \|\mathbf 1_{\Omega_j}M_mG|x|\|_{\ell^{p_2}}^s \lesssim
 M^s\|x\|_{\ell^{p_1}}^s
 \sum_{j\geq0}
 e^{-\frac{\alpha s}{8}|\ell-j|^2}
 e^{-\frac{\alpha s}{8}|j-k|^2}.
\]
Since
\[
 \sum_{j\geq0}
 e^{-\frac{\alpha s}{8}|\ell-j|^2}
 e^{-\frac{\alpha s}{8}|j-k|^2}
 \lesssim
 e^{-\frac{\alpha s}{16}|\ell-k|^2},
\]
we conclude that
\[
 \|\mathbf 1_{\Omega_\ell}C^\mu x\|_{\ell^{p_2}}
 \lesssim
 M e^{-\frac{\alpha}{16}|\ell-k|^2}
 \|x\|_{\ell^{p_1}}.
\]

\smallskip
\noindent
Finally, let $x^{(j)}$ and $k_j$ be as in the
statement. By passing to a sufficiently sparse
subsequence, the off-diagonal estimate just proved
gives, for $i\ne j$,
\[
 \|\mathbf 1_{\Omega_{k_j}}C^\mu x^{(i)}\|_{\ell^{p_2}}
 \leq
 2^{-\frac{1}{s}}\varepsilon_0.
\]
Hence
\[
 \|
 \mathbf 1_{\Omega_{k_j}}
 (C^\mu x^{(j)}-C^\mu x^{(i)})
 \|_{\ell^{p_2}}^s
 \geq
 \|\mathbf 1_{\Omega_{k_j}}C^\mu x^{(j)}\|_{\ell^{p_2}}^s
 -
 \|\mathbf 1_{\Omega_{k_j}}C^\mu x^{(i)}\|_{\ell^{p_2}}^s\geq
 \frac12\varepsilon_0^s.
\]
Since the $\ell^{p_2,q_2}$ quasi-norm dominates each shell
$\ell^{p_2}$ norm,
\[
 \|C^\mu x^{(j)}-C^\mu x^{(i)}\|_{\ell^{p_2,q_2}}
 \geq
 2^{-\frac{1}{s}}\varepsilon_0.
\]
Thus the vectors $C^\mu x^{(j)}$ are pairwise separated in
$\ell^{p_2,q_2}(\Omega)$.
\end{proof}

\smallskip
\noindent
For mappings into the outer endpoint, we also use the matrix
associated with the canonical dual-atom reconstruction. Assume
$q_2=\infty$. Then
\(
 \nu_2=-\frac1{p_2}.
\)
Whenever the integral is defined, set
\[
 D^\mu_{\gamma,\lambda}
 :=
 T_\mu^\alpha
 \big((1+|\lambda|)^{-\nu_1}\varphi_\lambda\big)(\gamma)
 e^{-\frac{\alpha}{2}|\gamma|^2}
 (1+|\gamma|)^{\nu_2},
 \qquad \gamma,\lambda\in\Omega.
\]
We retain the weighted local-mass sequence
$m=(m_\omega)_{\omega\in\Omega}$
defined in \eqref{eq-weighted-local-mass}.

\begin{proposition}[Endpoint reconstruction matrix]
\label{prop:endpoint-reconstruction}
Let $D^\mu$ and $m=(m_\omega)_{\omega\in\Omega}$ be as above.

\begin{enumerate}[label=\textup{(\alph*)}]
\item For every $L>0$,
\begin{equation}\label{eq:endpoint-entrywise}
 |D^\mu_{\gamma,\lambda}|
 \lesssim
 \sum_{\omega\in\Omega}
 e^{-\frac{\alpha}{8}|\gamma-\omega|^2}
 m_\omega
 (1+|\omega-\lambda|)^{-L},
 \qquad \gamma,\lambda\in\Omega.
\end{equation}
Equivalently, if
\( G_{\gamma,\omega}
 :=
 e^{-\frac{\alpha}{8}|\gamma-\omega|^2}\),
and \(K_L(\omega,\lambda)
 :=
 (1+|\omega-\lambda|)^{-L}\),
then
\begin{equation}\label{eq:endpoint-matrix-majorant}
 |D^\mu|\lesssim GM_mK_L .
\end{equation}

\item Suppose that $m\in\ell^{p_{12},\infty}(\Omega)$. Then, for $L$ sufficiently large, $D^\mu$ is absolutely defined on
$\ell^{p_1,q_1}(\Omega)$ and
\(
 D^\mu:
 \ell^{p_1,q_1}(\Omega)
 \longrightarrow
 \ell^{p_2,\infty}(\Omega)
\)
is bounded. More precisely,
\[
 |D^\mu x|
 \lesssim
 GM_mK_L|x|.
\]
Moreover, for every
$f\in\mathcal F_\alpha^{p_1,q_1}$,
\begin{equation}\label{eq:endpoint-action}
 R_{\alpha,\Omega}^{p_2,\infty}T_\mu^\alpha f
 =
 D^\mu R_{\alpha,\Omega}^{p_1,q_1}f.
\end{equation}
The matrix series in \eqref{eq:endpoint-action} converges absolutely
in each coordinate, and the corresponding Toeplitz integral is
absolutely convergent.
\end{enumerate}
Here $L$ is a fixed localization order, chosen sufficiently
large in part~(b). The implicit constants depend only on
$L$, $\alpha$, the fixed exponent pairs, and the fixed
sampling-lattice data.
\end{proposition}

\begin{proof}
(a) Fix $L>0$. By the rapid localization of the canonical dual atoms,
for every sufficiently large $N$,
\[
 |\varphi_\lambda(w)|
 e^{-\frac{\alpha}{2}|w|^2}
 \lesssim_N
 (1+|w-\lambda|)^{-N}.
\]
Consequently,
\[
 \begin{aligned}
 |D^\mu_{\gamma,\lambda}|
 &\lesssim
 (1+|\gamma|)^{\nu_2}
 (1+|\lambda|)^{-\nu_1}
 \int_{\mathbb C}
 e^{-\frac{\alpha}{2}|\gamma-w|^2}
 (1+|w-\lambda|)^{-N}\,d\mu(w).
 \end{aligned}
\]
Decompose the integral over the exact square partition
\[
 \mathbb C=\bigsqcup_{\omega\in\Omega}Q_\omega .
\]
Since $Q_\omega$ has uniformly bounded diameter, for
$w\in Q_\omega$,
\[
 e^{-\frac{\alpha}{2}|\gamma-w|^2}
 \lesssim
 e^{-\frac{\alpha}{4}|\gamma-\omega|^2},
 \qquad
 (1+|w-\lambda|)^{-N}
 \lesssim
 (1+|\omega-\lambda|)^{-N}.
\]
Hence
\[
 \begin{aligned}
 |D^\mu_{\gamma,\lambda}|
 &\lesssim
 (1+|\gamma|)^{\nu_2}
 (1+|\lambda|)^{-\nu_1}
 \sum_{\omega\in\Omega}
 e^{-\frac{\alpha}{4}|\gamma-\omega|^2}
 (1+|\omega-\lambda|)^{-N}
 \widehat\mu_{R_*}(\omega).
 \end{aligned}
\]
As in the proof of Proposition~\ref{prop:toeplitz-testing-matrix},
polynomial moderation gives
\[
 (1+|\gamma|)^{\nu_2}
 \lesssim
 (1+|\omega|)^{\nu_2}
 (1+|\gamma-\omega|)^{|\nu_2|}
\qquad \text{and} \qquad
 (1+|\lambda|)^{-\nu_1}
 \lesssim
 (1+|\omega|)^{-\nu_1}
 (1+|\omega-\lambda|)^{|\nu_1|}.
\]
The first polynomial loss is absorbed by the Gaussian, while the
second is absorbed by choosing $N > L + |\nu_1|$. Thus
\[
 |D^\mu_{\gamma,\lambda}|
 \lesssim
 \sum_{\omega\in\Omega}
 e^{-\frac{\alpha}{8}|\gamma-\omega|^2}
 m_\omega
 (1+|\omega-\lambda|)^{-L},
\]
which proves \eqref{eq:endpoint-entrywise} and
\eqref{eq:endpoint-matrix-majorant}.

\medskip
\noindent
(b) Choose $L$ sufficiently large. Given any prescribed $A>0$, $L$
may be chosen so that the residual polynomial kernel is summable on
the separated lattice, also after taking the fixed subunit powers
occurring below. If $\lambda\in\Omega_k$ and $\omega\in\Omega_j$,
then
\[
 |\omega-\lambda|
 \ge (|j-k|-1)_+.
\]
The same shell argument as in
Proposition~\ref{prop:gaussian-matrices} therefore gives
\[
 \|\mathbf1_{\Omega_j}
   K_L\mathbf1_{\Omega_k}x\|_{\ell^p}
 \lesssim
 (1+|j-k|)^{-A}
 \|x|_{\Omega_k}\|_{\ell^p}
\]
for every fixed exponent $0<p\le\infty$ occurring here. The implicit constant may depend on $A$, $L$, $p$,
and the fixed sampling-lattice geometry.
Convolution on the shell index yields
\[
 K_L:\ell^{p,q}(\Omega)\longrightarrow\ell^{p,q}(\Omega)
\]
for all relevant $p,q$, and
\[
 K_L:\ell_0^{p,\infty}(\Omega)
 \longrightarrow
 \ell_0^{p,\infty}(\Omega).
\]

\medskip
\noindent
Assume \(m \in \ell^{p_{12}, \infty}\). By
Proposition~\ref{prop:annular-multipliers},
\(
 M_m:
 \ell^{p_1,\infty}(\Omega)
 \longrightarrow
 \ell^{p_2,\infty}(\Omega)
\)
is bounded, with operator quasi-norm comparable to
\(
 \|m\|_{\ell^{p_{12}, \infty}}.
\)
Proposition~\ref{prop:gaussian-matrices} controls the final Gaussian
matrix $G$. If $x\in\ell^{p_1,q_1}(\Omega)$, then the boundedness
of $K_L$ established above gives
\[
 K_L|x|\in\ell^{p_1,q_1}(\Omega)
 \subset\ell^{p_1,\infty}(\Omega).
\]
Hence
\[
 GM_mK_L|x|\in\ell^{p_2,\infty}(\Omega).
\]
In particular, the right-hand side is finite in every coordinate.
The estimate used in the proof of part~(a), with absolute values
taken inside the Toeplitz integral, therefore shows first that the
entries of $D^\mu$ are absolutely defined and then that
\[
 |D^\mu x|
 \lesssim
 GM_mK_L|x|.
\]
Let now
\( x=R_{\alpha,\Omega}^{p_1,q_1}f.
\)
The canonical dual-atom reconstruction from
Subsection~\ref{subsec:lattice} gives
\[
 f
 =
 \sum_{\lambda\in\Omega}
 x_\lambda
 (1+|\lambda|)^{-\nu_1}\varphi_\lambda .
\]
For each $\gamma\in\Omega$, the preceding estimates before
integration give
\[
 \sum_{\lambda\in\Omega}
 |x_\lambda|(1+|\lambda|)^{-\nu_1}
 \int_{\mathbb C}
 |\varphi_\lambda(w)|
 e^{\alpha\operatorname{Re}(\gamma\overline w)
    -\alpha|w|^2-\frac{\alpha}{2}|\gamma|^2}
 (1+|\gamma|)^{\nu_2}\,d\mu(w) \lesssim
 (GM_mK_L|x|)_\gamma
 <\infty .
\]
Thus the sum of the absolute integrands is $\mu$-integrable.
Tonelli's theorem gives absolute convergence and permits
interchange of the dual-atom series and the Toeplitz integral. 
The same cell estimates, together with the
pointwise Fock growth bound and
$\mu(Q_\omega)\le
M(1+|\omega|)^{\nu_1-\nu_2}$,
also give locally uniform absolute convergence
of the Toeplitz integral for $z\in\mathbb C$.
Therefore
\[
 R_{\alpha,\Omega}^{p_2,\infty}T_\mu^\alpha f
 =
 D^\mu x
 =
 D^\mu R_{\alpha,\Omega}^{p_1,q_1}f,
\]
which proves \eqref{eq:endpoint-action}. In particular, no
norm-density argument is used when $q_1=\infty$.
\end{proof}

\section{Local masses and Gaussian Berezin profiles}
\label{sec:comparison}

This section passes between the continuous local-mass and Gaussian
Berezin criteria and their discrete counterparts, including their
asymptotic tails. The argument has three steps: first we remove the
dependence on the radius in the local-mass function, then discretize the
local masses on bounded-geometry cells, and finally compare the resulting
discrete profile with the Gaussian Berezin profile. For $R>0$ and $\sigma\in\mathbb R$, set
\[
 h_{R,\sigma}(z)
 :=
 (1+|z|)^\sigma \widehat\mu_R(z),
 \qquad
 g_{\alpha,\sigma}(z)
 :=
 (1+|z|)^\sigma \widetilde\mu_\alpha(z).
\]
We shall use exact bounded-geometry cell systems. By this we mean a
separated configuration $\Lambda\subset\mathbb C$ together with Borel
sets $(Q_\lambda)_{\lambda\in\Lambda}$ such that
\begin{equation}\label{eq:cell-system}
 \mathbb C=\bigsqcup_{\lambda\in\Lambda}Q_\lambda,
 \qquad
 B(\lambda,c_0)\subset Q_\lambda\subset B(\lambda,C_0),
 \qquad
 0<c_0<C_0<\infty.
\end{equation}
Both the pocket decomposition of Subsection~\ref{subsec:pocket} and the fixed
square cells associated with $\Omega$ in Subsection~\ref{subsec:lattice} have this property.

\begin{proposition}[Radius stability of local masses]
\label{prop:local-mass-radius}
Let $0<p,q\leq\infty$, $\sigma\in\mathbb R$, and let $\mu$ be a
positive locally finite Borel measure. For any fixed $R_1,R_2>0$,
\[
 \|h_{R_1,\sigma}\|_{L_{\rm sh}^{p,q}}
 \asymp
 \|h_{R_2,\sigma}\|_{L_{\rm sh}^{p,q}}
\quad \text{and} \quad
 \lim_{N\to\infty}
 \|\mathbf 1_{\{|z|\geq N\}}h_{R_1,\sigma}\|_{L_{\rm sh}^{p,q}}
 \asymp
 \lim_{N\to\infty}
 \|\mathbf 1_{\{|z|\geq N\}}h_{R_2,\sigma}\|_{L_{\rm sh}^{p,q}},
\]
whenever either limit is finite. 
The implicit constants depend only
on $p,q,\sigma,R_1,R_2$.
\end{proposition}

\begin{proof}
We first record the effect of a fixed translation on the shell
quasi-norm. If $h$ is measurable and $w\in\mathbb C$ is fixed, then
a translate of an annulus $A_k$ meets only finitely many annuli $A_j$
with $|j-k|\leq J$, where $J$ depends only on $|w|$. Hence the
powered quasi-triangle inequality and finite overlap give
\[
 \|h(\cdot+w)\|_{L_{\rm sh}^{p,q}}
 \lesssim_w
 \|h\|_{L_{\rm sh}^{p,q}}.
\]
The same argument applied to the truncated function gives
\[
 \|\mathbf 1_{\{|z|\geq N\}}h(z+w)\|_{L_{\rm sh}^{p,q}}
 \lesssim_w
 \|\mathbf 1_{\{|z|\geq N-|w|\}}h(z)\|_{L_{\rm sh}^{p,q}}.
\]
Cover $B(0,R_2)$ by finitely many translates
\[
 B(0,R_2)\subset\bigcup_{\ell=1}^L B(w_\ell,R_1).
\]
Then
\[
 \widehat\mu_{R_2}(z)
 \leq
 \sum_{\ell=1}^L
 \widehat\mu_{R_1}(z+w_\ell).
\]
Since, for each fixed $w_\ell$,
\(
 (1+|z|)^\sigma
 \asymp
 (1+|z+w_\ell|)^\sigma,
\)
the preceding translation estimate yields
\[
 \|h_{R_2,\sigma}\|_{L_{\rm sh}^{p,q}}
 \lesssim
 \|h_{R_1,\sigma}\|_{L_{\rm sh}^{p,q}}.
\]
Interchanging $R_1$ and $R_2$ proves the norm equivalence.

\smallskip
\noindent
Using the truncated translation estimate in the same argument gives
the corresponding comparison of the asymptotic tails. Interchanging
$R_1$ and $R_2$ completes the proof.
\end{proof}

Having removed the dependence on the choice of radius, we now pass
between continuous local masses and their values on a bounded-geometry
cell system.

\begin{proposition}[Cell discretization and Gaussian majorization]
\label{prop:cell-discretization}
Let $\alpha>0$, $0<p,q\leq\infty$, $\sigma\in\mathbb R$, and let
$\mu$ be a positive locally finite Borel measure. Assume that
\[
 \mathbb C=\bigsqcup_{\lambda\in\Lambda}Q_\lambda,
 \qquad
 B(\lambda,c_0)\subset Q_\lambda\subset B(\lambda,C_0),
\]
with $0<c_0<C_0<\infty$, and let $R>C_0$. Put
\[
 m_{R,\sigma}(\lambda)
 :=
 (1+|\lambda|)^\sigma\widehat\mu_R(\lambda).
\]
The following assertions hold:
\begin{enumerate}[label=\textup{(\alph*)}]
\item
\[
 \|m_{R,\sigma}\|_{\ell^{p,q}(\Lambda)}
 \asymp
 \|h_{R,\sigma}\|_{L_{\rm sh}^{p,q}}.
\]

\item
\[
 \lim_{N\to\infty}
 \|\mathbf 1_{\{|\lambda|\geq N\}}m_{R,\sigma}\|_{\ell^{p,q}(\Lambda)}
 \asymp
 \lim_{N\to\infty}
 \|\mathbf 1_{\{|z|\geq N\}}h_{R,\sigma}\|_{L_{\rm sh}^{p,q}},
\]
whenever either limit is finite.

\item
\[
 g_{\alpha,\sigma}(z)
 \lesssim
 \sum_{\lambda\in\Lambda}
 e^{-\frac{\alpha}{4}|z-\lambda|^2}
 m_{R,\sigma}(\lambda),
 \qquad z\in\mathbb C.
\]
\end{enumerate}
The implicit constants depend only on
$\alpha,p,q,\sigma,R$ and the fixed cell geometry.
\end{proposition}

\begin{proof}
We prove the three assertions in order.

\smallskip
\noindent
(a) If $z\in Q_\lambda$, then
\[
 B(z,R-C_0)
 \subset
 B(\lambda,R)
 \subset
 B(z,R+C_0),
\]
and
\[
 (1+|z|)^\sigma
 \asymp
 (1+|\lambda|)^\sigma.
\]
Consequently,
\begin{equation}\label{eq:cell-mass-comparison}
 h_{R-C_0,\sigma}(z)
 \lesssim
 m_{R,\sigma}(\lambda)
 \lesssim
 h_{R+C_0,\sigma}(z),
 \qquad z\in Q_\lambda.
\end{equation}
Define the step function
\[
 m_{R,\sigma}^{\sharp}(z)
 :=
 \sum_{\lambda\in\Lambda}
 m_{R,\sigma}(\lambda)\mathbf 1_{Q_\lambda}(z).
\]
The inclusions in \eqref{eq:cell-system} give uniform upper and lower
bounds for the areas of the cells. Moreover, if
$\lambda\in\Lambda_k$, then $Q_\lambda$ meets only finitely many
annuli $A_j$ with $|j-k|\leq J$, where $J$ depends only on $C_0$.
Thus finite-shell overlap, together with the powered
quasi-triangle inequality when necessary, gives
\[
 \|m_{R,\sigma}^{\sharp}\|_{L_{\rm sh}^{p,q}}
 \asymp
 \|m_{R,\sigma}\|_{\ell^{p,q}(\Lambda)}.
\]
Combining this with \eqref{eq:cell-mass-comparison} and
Proposition~\ref{prop:local-mass-radius} yields
\[
 \|m_{R,\sigma}\|_{\ell^{p,q}(\Lambda)}
 \asymp
 \|h_{R,\sigma}\|_{L_{\rm sh}^{p,q}}.
\]

\smallskip
\noindent
(b) If $z\in Q_\lambda$, then
\(
 \big||z|-|\lambda|\big|\leq C_0.
\)
Hence restricting $m_{R,\sigma}^{\sharp}$ to $\{|z|\geq N\}$ and
restricting its coefficient sequence to $\{|\lambda|\geq N\}$ change
the truncation threshold by at most a fixed amount. Applying the
finite-shell comparison from part~(a) to the truncated step functions
and sequences, and using Proposition~\ref{prop:local-mass-radius} for
the local-mass terms, gives
\[
 \lim_{N\to\infty}
 \|\mathbf 1_{\{|\lambda|\geq N\}}m_{R,\sigma}
 \|_{\ell^{p,q}(\Lambda)}
 \asymp
 \lim_{N\to\infty}
 \|\mathbf 1_{\{|z|\geq N\}}h_{R,\sigma}
 \|_{L_{\rm sh}^{p,q}},
\]
whenever either limit is finite.

\smallskip
\noindent
(c) For $w\in Q_\lambda$, the bounded diameter of $Q_\lambda$ gives
\(
 |w-\lambda|\leq C_0.
\)
Together with polynomial moderation, this yields
\[
 (1+|z|)^\sigma e^{-\alpha|z-w|^2}
 \lesssim
 (1+|\lambda|)^\sigma
 e^{-\frac{\alpha}{4}|z-\lambda|^2}.
\]
Using the exact cell decomposition, we therefore obtain
\[
 \begin{aligned}
 g_{\alpha,\sigma}(z)
 &=
 (1+|z|)^\sigma
 \sum_{\lambda\in\Lambda}
 \int_{Q_\lambda}
 e^{-\alpha|z-w|^2}\,d\mu(w)\\
 &\lesssim
 \sum_{\lambda\in\Lambda}
 e^{-\frac{\alpha}{4}|z-\lambda|^2}
 (1+|\lambda|)^\sigma\mu(Q_\lambda) \leq
 \sum_{\lambda\in\Lambda}
 e^{-\frac{\alpha}{4}|z-\lambda|^2}
 m_{R,\sigma}(\lambda),
 \end{aligned}
\]
because $Q_\lambda\subset B(\lambda,R)$. This proves~(c).
\end{proof}

The preceding discretization and Gaussian majorization now give the
equivalence of the two continuous profiles.

\begin{theorem}[Local masses and Gaussian Berezin profiles]
\label{thm:local-mass-berezin}
Let $\mu$ be a positive locally finite Borel measure, let
$\alpha,R>0$, $0<p,q\leq\infty$, and $\sigma\in\mathbb R$. Then
\[
 \|h_{R,\sigma}\|_{L_{\rm sh}^{p,q}}
 \asymp
 \|g_{\alpha,\sigma}\|_{L_{\rm sh}^{p,q}}
\qquad \text{and} \qquad
 \lim_{N\to\infty}
 \|\mathbf 1_{\{|z|\geq N\}}h_{R,\sigma}\|_{L_{\rm sh}^{p,q}}
 \asymp
 \lim_{N\to\infty}
 \|\mathbf 1_{\{|z|\geq N\}}g_{\alpha,\sigma}\|_{L_{\rm sh}^{p,q}},
\]
whenever either limit is finite. In particular, one tail vanishes
if and only if the other does. 
The implicit constants in both comparisons depend only on $\alpha,R,p,q,\sigma$ and the fixed
geometric discretization.
\end{theorem}

\begin{proof}
By Proposition~\ref{prop:local-mass-radius}, we may work at the fixed
radius $R_*$ and with the square cell system
$(Q_\omega)_{\omega\in\Omega}$ from Subsection~\ref{subsec:lattice}. Put
\[
 m_\omega
 :=
 (1+|\omega|)^\sigma\widehat\mu_{R_*}(\omega).
\]
For $\tau>0$ and a nonnegative sequence
$d=(d_\omega)_{\omega\in\Omega}$, set
\[
 \mathcal G_\tau d(z)
 :=
 \sum_{\omega\in\Omega}e^{-\tau|z-\omega|^2}d_\omega.
\]
If $z\in Q_\gamma$, then the bounded diameter of the square cells
gives
\[
 \mathcal G_\tau d(z)
 \lesssim
 \sum_{\omega\in\Omega}
 e^{-\frac{\tau}{2}|\gamma-\omega|^2}d_\omega.
\]
Proposition~\ref{prop:gaussian-matrices}, together with the
cell-to-shell comparison, therefore yields
\begin{equation}\label{eq:gaussian-extension}
 \|\mathcal G_\tau d\|_{L_{\rm sh}^{p,q}}
 \lesssim_\tau
 \|d\|_{\ell^{p,q}}.
\end{equation}
We also record the corresponding tail estimate. Let $N,L\geq1$ and
split
\[
 d
 =
 \mathbf 1_{\{|\omega|\geq N-L\}}d
 +
 \mathbf 1_{\{|\omega|<N-L\}}d.
\]
The first part is controlled by \eqref{eq:gaussian-extension}. For
the second part, if $|z|\geq N$ and $|\omega|<N-L$, then
$|z-\omega|\geq L$. Hence
\[
 e^{-\tau|z-\omega|^2}
 \leq
 e^{-\frac{\tau}{2}L^2}
 e^{-\frac{\tau}{2}|z-\omega|^2}.
\]
Applying \eqref{eq:gaussian-extension} with $\frac{\tau}{2}$ gives
\begin{equation}\label{eq:gaussian-extension-tail}
 \|\mathbf 1_{\{|z|\geq N\}}\mathcal G_\tau d\|_{L_{\rm sh}^{p,q}}
 \lesssim_\tau
 \|\mathbf 1_{\{|\omega|\geq N-L\}}d\|_{\ell^{p,q}}
 +
 e^{-\frac{\tau}{2}L^2}\|d\|_{\ell^{p,q}}.
\end{equation}

\smallskip
\noindent
Proposition~\ref{prop:cell-discretization}(c), with
$\tau=\frac{\alpha}{4}$, gives
\[
 g_{\alpha,\sigma}
 \lesssim
 \mathcal G_{\frac{\alpha}{4}}m.
\]
Hence, by \eqref{eq:gaussian-extension} and
Proposition~\ref{prop:cell-discretization}(a),
\[
 \|g_{\alpha,\sigma}\|_{L_{\rm sh}^{p,q}}
 \lesssim
 \|m\|_{\ell^{p,q}}
 \asymp
 \|h_{R_*,\sigma}\|_{L_{\rm sh}^{p,q}}
 \asymp
 \|h_{R,\sigma}\|_{L_{\rm sh}^{p,q}}.
\]
Conversely, if $w\in B(z,R)$, then
\(
 e^{-\alpha|z-w|^2}\geq e^{-\alpha R^2},
\)
and therefore
\begin{equation}\label{eq-local-mass-berezin-pointwise}
\widehat\mu_R(z)
\le e^{\alpha R^2}\widetilde\mu_\alpha(z),
\qquad z\in\mathbb C,\quad R>0.
\end{equation}
Thus
\(
 h_{R,\sigma}(z)
 \leq
 e^{\alpha R^2}g_{\alpha,\sigma}(z),
\)
which proves the first assertion.

\smallskip
\noindent
Let
\[
 E_R
 :=
 \lim_{N\to\infty}
 \|\mathbf 1_{\{|z|\geq N\}}h_{R,\sigma}\|_{L_{\rm sh}^{p,q}},
\qquad \text{and} \qquad
 E_B
 :=
 \lim_{N\to\infty}
 \|\mathbf 1_{\{|z|\geq N\}}g_{\alpha,\sigma}\|_{L_{\rm sh}^{p,q}}.
\]
The preceding pointwise estimate immediately gives
\[
 E_R\lesssim E_B.
\]
Suppose that either limit is finite. Then $E_R<\infty$, and
Proposition~\ref{prop:local-mass-radius} shows that the
corresponding $R_*$-tail is finite. Since $\mu$ is locally
finite, $h_{R_*,\sigma}$ has finite shell quasi-norm on every
bounded region. Hence
\(
  h_{R_*,\sigma}\in L_{\rm sh}^{p,q}.
\)
By Proposition~\ref{prop:cell-discretization}(a),
\[
 m\in\ell^{p,q}(\Omega).
\]
Using Proposition~\ref{prop:cell-discretization}(c) and
\eqref{eq:gaussian-extension-tail} with $\tau= \frac{\alpha}{4}$, we obtain
\[
 \|\mathbf 1_{\{|z|\geq N\}}g_{\alpha,\sigma}\|_{L_{\rm sh}^{p,q}}
 \lesssim
 \|\mathbf 1_{\{|\omega|\geq N-L\}}m
 \|_{\ell^{p,q}}
 +
 e^{-\frac{\alpha}{8}L^2}
 \|m\|_{\ell^{p,q}}.
\]
Letting $N\to\infty$ and using
Proposition~\ref{prop:cell-discretization}(b) together with
Proposition~\ref{prop:local-mass-radius}, we get
\[
 E_B
 \lesssim
 E_R
 +
 e^{-\frac{\alpha}{8}L^2}
 \|m\|_{\ell^{p,q}}.
\]
Finally, letting $L\to\infty$ yields
\(
 E_B\lesssim E_R.
\)
Together with the reverse estimate, this proves the second assertion.
\end{proof}

\section{Proofs of the main operator results and particular cases}
\label{sec:operators}

We now combine the analytic discretizations of Section~\ref{sec:tools}, the
sequence-space mechanisms of Section~\ref{sec:discrete}, and the profile comparisons
of Section~\ref{sec:comparison}. We first prove the boundedness criterion between the
large mixed-norm spaces in the full exponent range. We then turn to
compactness and the essential norm, where the cases $q_1\leq q_2$
and $q_2<q_1$ separate naturally. Mappings involving the little outer
endpoint are treated afterward, and the final subsection records several particular cases.

\subsection{Boundedness between the large mixed-norm spaces}
\label{subsec:large-boundedness}

By Theorem~\ref{thm:local-mass-berezin}, the local-mass and
Gaussian Berezin conditions in Theorems~\ref{thm:q1leq-q2}
and~\ref{thm:q2lt-q1} are equivalent. It therefore suffices to
characterize boundedness in terms of the local-mass condition. We do
this simultaneously in the two exponent ranges.
Thus throughout this subsection
\[
 0<p_1,q_1,p_2,q_2\leq\infty,
 \qquad
 \frac1{p_{12}}
 =
 \left(\frac1{p_2}-\frac1{p_1}\right)_+,
 \qquad
 \frac1{q_{12}}
 =
 \left(\frac1{q_2}-\frac1{q_1}\right)_+.
\]

\medskip
\noindent
\textit{Step 1: Discretization and Gaussian reduction.}
Suppose first that, for some fixed $R>0$,
\[
 (1+|\cdot|)^{\nu_2-\nu_1}\widehat\mu_R
 \in L_{\rm sh}^{p_{12},q_{12}}.
\]
By Proposition~\ref{prop:local-mass-radius}, we may work at the radius
$R_*$ fixed in~\eqref{eq-R*}. For the pocket centers $\zeta_Q$ of
Subsection~\ref{subsec:pocket}, put
\[
 m_Q
 :=
 (1+|\zeta_Q|)^{\nu_2-\nu_1}
 \widehat\mu_{R_*}(\zeta_Q).
\]
Since $R_*>C_{\rm out}\delta$, Proposition~\ref{prop:cell-discretization}(a),
applied to the pocket cell system, gives
\begin{equation}\label{eq:5-pocket-mass}
 \|m\|_{\ell^{p_{12},q_{12}}(\mathcal P_\delta)}
 \lesssim
 \big\|
 (1+|\cdot|)^{\nu_2-\nu_1}
 \widehat\mu_{R_*}
 \big\|_{L_{\rm sh}^{p_{12},q_{12}}}.
\end{equation}
Let $f\in\mathcal F_\alpha^{p_1,q_1}$ and set
\[
 d:=A_\delta^{p_1,q_1}f.
\]
By the pocket recovery of Subsection~\ref{subsec:pocket},
\[
 \|d\|_{\ell^{p_1,q_1}(\mathcal P_\delta)}
 \lesssim
 \|f\|_{\mathcal F_\alpha^{p_1,q_1}},
 \qquad
 f
 =
 \sum_{Q\in\mathcal P_\delta}
 d_Q\,\widetilde\kappa_{\alpha,\zeta_Q}.
\]
Since
\(
 \|\kappa_{\alpha,\zeta_Q}\|_{\mathcal F_\alpha^{p_1,q_1}}
 \asymp
 (1+|\zeta_Q|)^{\nu_1}
\)
and
\(
 |\kappa_{\alpha,\zeta_Q}(w)|
 e^{-\frac{\alpha}{2}|w|^2}
 =
 e^{-\frac{\alpha}{2}|w-\zeta_Q|^2},
\)
we obtain
\begin{equation}\label{eq:5-source-envelope}
 |f(w)|e^{-\frac{\alpha}{2}|w|^2}
 \lesssim
 \sum_{Q\in\mathcal P_\delta}
 |d_Q|
 (1+|\zeta_Q|)^{-\nu_1}
 e^{-\frac{\alpha}{2}|w-\zeta_Q|^2}.
\end{equation}
The series on the right converges locally uniformly: the coefficients
are bounded by the sequence quasi-norm of $d$, the polynomial factor
has at most fixed polynomial growth, and the pocket centers have
bounded geometry.

\smallskip
\noindent
For $z\in\mathbb C$, consider the possibly infinite positive integral
\[
 I_f(z)
 :=
 \frac{\alpha}{\pi}
 \int_{\mathbb C}
 |f(w)|
 e^{\alpha\operatorname{Re}(z\overline w)-\alpha|w|^2}
 \,d\mu(w).
\]
Then using \eqref{eq:5-source-envelope}, Tonelli's theorem, and
the exact pocket decomposition, we obtain
\[
 I_f(z)e^{-\frac{\alpha}{2}|z|^2}(1+|z|)^{\nu_2}
 \lesssim
 \sum_{P,Q\in\mathcal P_\delta}
 e^{-\frac{\alpha}{8}|z-\zeta_P|^2}
 m_P
 e^{-\frac{\alpha}{8}|\zeta_P-\zeta_Q|^2}
 |d_Q|.
\]
Indeed, on the pocket $P$ the two polynomial factors may be transferred
to $\zeta_P$, since
\(
 (-\nu_1)+\nu_2=\nu_2-\nu_1,
\)
and the resulting fixed polynomial losses are absorbed by slightly
weakening the Gaussian factors.

\smallskip
\noindent
Let
\[
 G_{P,Q}:=e^{-\frac{\alpha}{8}|\zeta_P-\zeta_Q|^2},
 \qquad
 y:=M_mG|d|.
\]
Then
\begin{equation}\label{eq:5-target-envelope}
 I_f(z)e^{-\frac{\alpha}{2}|z|^2}(1+|z|)^{\nu_2}
 \lesssim
 \sum_{P\in\mathcal P_\delta}
 e^{-\frac{\alpha}{8}|z-\zeta_P|^2}y_P.
\end{equation}
Propositions~\ref{prop:gaussian-matrices} and
\ref{prop:annular-multipliers} give
\begin{equation}\label{eq:5-y-bound}
 \|y\|_{\ell^{p_2,q_2}}
 \lesssim
 \|m\|_{\ell^{p_{12},q_{12}}}
 \|d\|_{\ell^{p_1,q_1}}.
\end{equation}

\medskip
\noindent
\textit{Step 2: Well-definedness and the target estimate.}
Restrict the integral defining $I_f(z)$ to $|w|>N$. If $w\in P$,
then
\[
 |\zeta_P|
 \geq
 |w|-|w-\zeta_P|
 >
 N-C_{\rm out}\delta.
\]
Hence~\eqref{eq:5-target-envelope} bounds this tail by
\[
 \sum_{|\zeta_P|>N-C_{\rm out}\delta}
 e^{-\frac{\alpha}{8}|z-\zeta_P|^2}y_P.
\]
Every coordinate $y_P$ is bounded by a fixed multiple of
$\|y\|_{\ell^{p_2,q_2}}$. Since the pocket centers have bounded
geometry, Gaussian decay therefore gives, for every compact
$K\subset\mathbb C$,
\[
 \sup_{z\in K}
 \sum_{|\zeta_P|>N-C_{\rm out}\delta}
 e^{-\frac{\alpha}{8}|z-\zeta_P|^2}y_P
 \longrightarrow0
 \qquad (N\to\infty).
\]
Thus the absolute integral converges locally uniformly in $z$.

\smallskip
\noindent
For each $N$, the integral restricted to $|w|\leq N$ is entire:
on every compact set of $z$ the integrand and all its $z$-derivatives
are uniformly dominated, since $f$ is bounded on $B(0,N)$ and
$\mu(B(0,N))<\infty$.
These truncations converge locally
uniformly, and hence their limit is an entire function. Consequently
the Toeplitz integral defines $T_\mu^\alpha f$ and
\begin{equation}\label{eq:5-T-bounded-by-I}
 |T_\mu^\alpha f(z)|\leq I_f(z),
 \qquad z\in\mathbb C.
\end{equation}
Next, put
\[
 a_k
 :=
 \|y|_{\mathcal P_{\delta,k}}\|_{\ell^{p_2}},
 \qquad k\geq0.
\]
The one-shell Gaussian estimate from the proofs of
\cite[Lemmas~5.13 and~5.22]{FT}, applied to the Gaussian family
$e^{-\frac{\alpha}{8}|\cdot-\zeta_P|^2}$, gives
\[
 \sup_{\ell\leq r<\ell+1}
 M_{p_2}\!\left(
   \sum_{P\in\mathcal P_{\delta,k}}
   e^{-\frac{\alpha}{8}|\cdot-\zeta_P|^2}y_P,r
 \right)
 \lesssim
 (1+k)^{-\frac1{p_2}}
 e^{-\frac{\alpha}{32}|\ell-k|^2}a_k,
\]
with the usual interpretation when $p_2=\infty$. For $0<p_2<1$, the same estimate follows from the corresponding
$p_2$-subadditive one-shell argument in the proof of
\cite[Lemma~5.22]{FT}.

\smallskip
\noindent
Set
\[
 s:=\min\{1,p_2\}.
\]
Decomposing the Gaussian sum according to the annular blocks, using
Minkowski's inequality when $p_2\geq1$ and $p_2$-subadditivity when
$p_2<1$, and absorbing
\(
 \left(\frac{1+\ell}{1+k}\right)^{\frac{1}{p_2}}
\)
into a slightly weaker Gaussian, we obtain
\begin{equation}\label{eq:5-shell-envelope}
 \left[
 (1+\ell)^{\frac{1}{p_2}}
 \sup_{\ell\leq r<\ell+1}
 M_{p_2}\!\left(
   \sum_{P\in\mathcal P_\delta}
   e^{-\frac{\alpha}{8}|\cdot-\zeta_P|^2}y_P,r
 \right)
 \right]^s
 \lesssim
 \sum_{k\geq0}
 e^{-\frac{\alpha s}{64}|\ell-k|^2}a_k^s .
\end{equation}
By~\eqref{eq:5-T-bounded-by-I} and
\eqref{eq:5-target-envelope},
\[
 B_{p_2}(T_\mu^\alpha f;\ell)
 \lesssim
 (1+\ell)^{-\nu_2}
 \sup_{\ell\leq r<\ell+1}
 M_{p_2}\!\left(
   \sum_{P\in\mathcal P_\delta}
   e^{-\frac{\alpha}{8}|\cdot-\zeta_P|^2}y_P,r
 \right).
\]
Since
\(
 \frac1{q_2}-\nu_2=\frac1{p_2},
\)
the annular factor $(1+\ell)^{\frac{1}{q_2}}$ appearing in
\eqref{eq:annular} converts the preceding right-hand side exactly
into the quantity on the left of~\eqref{eq:5-shell-envelope}.

\smallskip
\noindent
To complete the outer estimate, put
\[
 b_\ell
 :=
 (1+\ell)^{\frac{1}{p_2}}
 \sup_{\ell\leq r<\ell+1}
 M_{p_2}\!\left(
   \sum_{P\in\mathcal P_\delta}
   e^{-\frac{\alpha}{8}|\cdot-\zeta_P|^2}y_P,r
 \right).
\]
Then~\eqref{eq:5-shell-envelope} gives
\[
 b_\ell^s
 \lesssim
 \sum_{k\geq0}
 h_{\ell-k}a_k^s,
 \qquad
 h_j:=e^{-\frac{\alpha s}{64}|j|^2}.
\]
If $q_2<\infty$, then
\[
 \|b\|_{\ell^{q_2}}^s
 =
 \|(b_\ell^s)_{\ell \geq 0}\|_{\ell^{\frac{q_2}{s}}}.
\]
When $\frac{q_2}{s}\geq1$, Young's inequality gives
\[
 \|b\|_{\ell^{q_2}}
 \lesssim
 \|a\|_{\ell^{q_2}},
\]
while for $\frac{q_2}{s}<1$ the same conclusion follows from the
$\frac{q_2}{s}$-subadditivity and the summability of every positive power
of $h$. If $q_2=\infty$, one simply takes suprema. Hence
\[
 \|b\|_{\ell^{q_2}}
 \lesssim
 \|a\|_{\ell^{q_2}}
 =
 \|y\|_{\ell^{p_2,q_2}}.
\]
The annular discretization therefore yields
\begin{equation}\label{eq:5-envelope-fock}
 \|T_\mu^\alpha f\|_{\mathcal F_\alpha^{p_2,q_2}}
 \lesssim
 \|y\|_{\ell^{p_2,q_2}}.
\end{equation}
Combining~\eqref{eq:5-pocket-mass},
\eqref{eq:5-y-bound}, and~\eqref{eq:5-envelope-fock}, we conclude that
\[ \|T_\mu^\alpha\|_{\mathcal F_\alpha^{p_1,q_1}
 \to\mathcal F_\alpha^{p_2,q_2}}
 \lesssim
 \big\|
 (1+|\cdot|)^{\nu_2-\nu_1}
 \widehat\mu_{R_*}
 \big\|_{L_{\rm sh}^{p_{12},q_{12}}}.
\]
This proves sufficiency and the upper norm estimate.

\medskip
\noindent
\textit{Step 3: Necessity via the square testing matrix.}
Conversely, suppose that
\(
 T_\mu^\alpha:
 \mathcal F_\alpha^{p_1,q_1}
 \longrightarrow
 \mathcal F_\alpha^{p_2,q_2}
\)
is bounded. Consider
\[
 B
 :=
 R_{\alpha,\Omega}^{p_2,q_2}
 T_\mu^\alpha
 S_{\alpha,\Omega}^{p_1,q_1}
 :
 \ell^{p_1,q_1}(\Omega)
 \longrightarrow
 \ell^{p_2,q_2}(\Omega).
\]
On finitely supported sequences its matrix is the square testing
matrix $C^\mu$ of Subsection~\ref{subsec:toeplitz-matrices}. By
Proposition~\ref{prop:toeplitz-testing-matrix},
\[
 C^\mu_{\lambda,\lambda}
 =
 \frac{\alpha}{\pi}
 (1+|\lambda|)^{\nu_2-\nu_1}
 \widetilde\mu_\alpha(\lambda).
\]
Lemma~\ref{lem:diagonal-extraction} and
Proposition~\ref{prop:annular-multipliers} therefore give
\begin{equation}\label{eq:5-diagonal-necessity}
 \left\|
 \big(
 (1+|\lambda|)^{\nu_2-\nu_1}
 \widetilde\mu_\alpha(\lambda)
 \big)_{\lambda\in\Omega}
 \right\|_{\ell^{p_{12},q_{12}}}
 \lesssim
 \|T_\mu^\alpha\|_{\mathcal F_\alpha^{p_1,q_1}
 \to\mathcal F_\alpha^{p_2,q_2}}.
\end{equation}
Only finitely supported inputs are used before diagonal extraction,
so no norm-density argument at the large outer endpoint is involved.

\medskip
\noindent
\textit{Step 4: Return to the continuous profiles.}
By \eqref{eq-local-mass-berezin-pointwise},
applied with $R=R_*$, the estimate~\eqref{eq:5-diagonal-necessity} gives
\[
 \left\|
 \big(
 (1+|\lambda|)^{\nu_2-\nu_1}
 \widehat\mu_{R_*}(\lambda)
 \big)_{\lambda\in\Omega}
 \right\|_{\ell^{p_{12},q_{12}}}
 \lesssim
 \|T_\mu^\alpha\|_{\mathcal F_\alpha^{p_1,q_1}
 \to\mathcal F_\alpha^{p_2,q_2}}.
\]
Since $R_*>\Delta_\Omega$, Proposition~\ref{prop:cell-discretization}(a),
now applied to the square cell system associated with $\Omega$, yields
\[
 \big\|
 (1+|\cdot|)^{\nu_2-\nu_1}
 \widehat\mu_{R_*}
 \big\|_{L_{\rm sh}^{p_{12},q_{12}}}
 \lesssim
 \|T_\mu^\alpha\|_{\mathcal F_\alpha^{p_1,q_1}
 \to\mathcal F_\alpha^{p_2,q_2}}.
\]
Proposition~\ref{prop:local-mass-radius} removes the choice of radius,
and Theorem~\ref{thm:local-mass-berezin} gives the equivalent Gaussian
Berezin condition. Consequently, for every fixed $R>0$,
\[
 \|T_\mu^\alpha\|_{\mathcal F_\alpha^{p_1,q_1}
 \to\mathcal F_\alpha^{p_2,q_2}}
 \asymp
 \big\|
 (1+|\cdot|)^{\nu_2-\nu_1}
 \widehat\mu_R
 \big\|_{L_{\rm sh}^{p_{12},q_{12}}}
 \asymp
 \big\|
 (1+|\cdot|)^{\nu_2-\nu_1}
 \widetilde\mu_\alpha
 \big\|_{L_{\rm sh}^{p_{12},q_{12}}}.
\]
Since $q_{12}=\infty$ when $q_1\leq q_2$, whereas $q_{12}<\infty$
when $q_2<q_1$, this proves the required local-mass characterization
in both exponent ranges. Together with
Theorem~\ref{thm:local-mass-berezin}, it completes the boundedness
parts and the asserted norm estimates in
Theorems~\ref{thm:q1leq-q2} and~\ref{thm:q2lt-q1}.

\subsection{Compactness and essential norm}
\label{subsec:compactness-essential}

For a bounded linear map $A:X\to Y$, set
\[
 \|A\|_{\rm ess}
 :=
 \inf\{
   \|A-K\|_{X\to Y}:
   K:X\to Y \text{ is compact and linear}
 \}.
\]
Throughout this subsection, assume that
\[
 T_\mu^\alpha:
 \mathcal F_\alpha^{p_1,q_1}
 \longrightarrow
 \mathcal F_\alpha^{p_2,q_2}
\]
is bounded. By Proposition~\ref{prop:local-mass-radius} and
Theorem~\ref{thm:local-mass-berezin}, the asymptotic local-mass
and Gaussian Berezin tails are independent, up to equivalence, of
the fixed radius and of the choice of profile. We may therefore work
at the radius $R_*$ fixed in~\eqref{eq-R*}. Put
\[
 h_*(z)
 :=
 (1+|z|)^{\nu_2-\nu_1}\widehat\mu_{R_*}(z)
\qquad
\text{and}
\qquad
 E_*
 :=
 \lim_{N\to\infty}
 \big\|
 \mathbf 1_{\{|z|\geq N\}}h_*
 \big\|_{L_{\rm sh}^{p_{12},q_{12}}}.
\]
The limit exists by monotonicity.

\medskip
\noindent
\textit{Step 1: Compactly supported symbols and tail truncation.}
Suppose first that $\mu$ is supported in a fixed disk $B(0,R_0)$.
Then $\mu(B(0,R_0))<\infty$. Uniform point evaluations on this disk
and the finiteness of $\mu$ give
\[
 \left|
 \int_{\mathbb C}
 f(w)\overline w^{\,n}e^{-\alpha|w|^2}\,d\mu(w)
 \right|
 \lesssim_{R_0,\mu}
 R_0^n\,
 \|f\|_{\mathcal F_\alpha^{p_1,q_1}},
 \qquad n\geq0.
\]
For $J\geq0$, define
\[
 T_{\mu,J}^\alpha f(z)
 :=
 \frac{\alpha}{\pi}
 \sum_{n=0}^J
 \frac{\alpha^n z^n}{n!}
 \int_{\mathbb C}
 f(w)\overline w^{\,n}
 e^{-\alpha|w|^2}\,d\mu(w).
\]
Its range is contained in the polynomials of degree at most $J$, so
$T_{\mu,J}^\alpha$ has finite rank.

\smallskip
\noindent
Set
\[
 s_2:=s(p_2,q_2)=\min\{1,p_2,q_2\}.
\]
The monomial estimate and Stirling's formula give
\[
 \|z^n\|_{\mathcal F_\alpha^{p_2,q_2}}
 \asymp 
 \alpha^{-\frac{n}{2}}\sqrt{n!}\,(1+n)^{\frac{1}{2q_2} - \frac{1}{4}}.
\]
Hence the powered
quasi-triangle inequality yields
\[
 \|T_\mu^\alpha-T_{\mu,J}^\alpha\|^{s_2}
 \lesssim_{R_0,\mu}
 \sum_{n>J}
 \left[
   \frac{\alpha^{\frac{n}{2}}R_0^n}{\sqrt{n!}}
   (1+n)^{\frac{1}{2q_2} - \frac{1}{4}}
 \right]^{s_2}
 \longrightarrow0.
\]
Thus compactly supported positive symbols give compact operators.

\smallskip
\noindent
For a general bounded $T_\mu^\alpha$, write
\[
 \mu_N:=\mathbf 1_{B(0,N)}\mu,
 \qquad
\mu_N^{\mathrm{tail}}:=\mu-\mu_N.
\]
The operator $T_{\mu_N}^\alpha$ is compact by the preceding argument.
Moreover,
\[
 \widehat{\mu_N^{\mathrm{tail}}}_{R_*}(z)
 \leq
 \mathbf 1_{\{|z|\geq N-R_*\}}
 \widehat\mu_{R_*}(z).
\]
The boundedness estimate from
Subsection~\ref{subsec:large-boundedness} therefore gives
\[
 \|T_\mu^\alpha\|_{\rm ess}
 \leq
 \|T_{\mu_N^{\mathrm{tail}}}^\alpha\|_{\mathcal F_\alpha^{p_1,q_1}
 \to\mathcal F_\alpha^{p_2,q_2}}
 \lesssim
 \big\|
 \mathbf 1_{\{|z|\geq N-R_*\}}h_*
 \big\|_{L_{\rm sh}^{p_{12},q_{12}}}.
\]
Letting $N\to\infty$, we obtain
\begin{equation}\label{eq:5-essential-upper}
 \|T_\mu^\alpha\|_{\rm ess}
 \lesssim E_*.
\end{equation}
In particular, $E_*=0$ implies compactness.

\medskip
\noindent
\textit{Step 2: The case $q_2<q_1$.}
In this range,
\(
 q_{12}<\infty.
\)
By Subsection~\ref{subsec:large-boundedness},
\[
 h_*\in L_{\rm sh}^{p_{12},q_{12}},
\]
and therefore its outer $\ell^{q_{12}}$-tail tends to zero. Hence
$E_*=0$, and~\eqref{eq:5-essential-upper} shows that every bounded
operator
\[
 T_\mu^\alpha:
 \mathcal F_\alpha^{p_1,q_1}
 \longrightarrow
 \mathcal F_\alpha^{p_2,q_2}
\]
is compact. Since compactness implies boundedness, this completes the
compactness assertion in Theorem~\ref{thm:q2lt-q1}.

\medskip
\noindent
\textit{Step 3: The case $q_1\leq q_2$.}
We now have
\(
 q_{12}=\infty.
\)
Consequently,
\[
 E_*
 =
 \limsup_{k\to\infty}
 \|h_*\|_{L^{p_{12}}(A_k,dA)}.
\]
The upper essential-norm estimate has already been obtained in
Step~1. It remains to prove the reverse inequality.

\smallskip
\noindent
Let $m=(m_\lambda)_{\lambda\in\Omega}$ be the
weighted local-mass sequence defined in
\eqref{eq-weighted-local-mass}.
Proposition~\ref{prop:cell-discretization}(b), applied to the square
cell system associated with $\Omega$, gives
\begin{equation}\label{eq:5-discrete-tail}
 E_*
 \asymp
 \limsup_{k\to\infty}\|m|_{\Omega_k}\|_{\ell^{p_{12}}}.
\end{equation}
Moreover,
\[
 M:=\sup_{k\geq0}\|m|_{\Omega_k}\|_{\ell^{p_{12}}}<\infty.
\]
If $E_*=0$, there is nothing further to prove.
Assume therefore that $E_*>0$. By \eqref{eq:5-discrete-tail}, there exist
a constant $c_*>0$, depending only on the fixed
structural and geometric parameters, and shell indices
$k_j\to\infty$ such that
\[
\|m|_{\Omega_{k_j}}\|_{\ell^{p_{12}}}
\ge c_*E_*,
\qquad j\ge1.
\]
Let
\[
 d_\lambda
 :=
 C^\mu_{\lambda,\lambda}
 =
 \frac{\alpha}{\pi}
 (1+|\lambda|)^{\nu_2-\nu_1}
 \widetilde\mu_\alpha(\lambda).
\]
By \eqref{eq-local-mass-berezin-pointwise},
applied with $R=R_*$, we have
\[
 \|d|_{\Omega_{k_j}}\|_{\ell^{p_{12}}}
 \gtrsim E_*.
\]
By the one-shell multiplier identity in
Proposition~\ref{prop:annular-multipliers}, for each $j$ there is
$u^{(j)}$ supported in $\Omega_{k_j}$ such that
\[
 \|u^{(j)}\|_{\ell^{p_1}}\leq1,
 \qquad
 \|M_du^{(j)}\|_{\ell^{p_2}}
 \gtrsim E_*.
\]
Applying the finite-set assertion of
Lemma~\ref{lem:diagonal-extraction}, and absorbing the resulting
signs into $u^{(j)}$, we obtain vectors $x^{(j)}$ supported in
$\Omega_{k_j}$ satisfying
\[ \|x^{(j)}\|_{\ell^{p_1}}\leq1,
 \qquad
 \|\mathbf 1_{\Omega_{k_j}}C^\mu x^{(j)}\|_{\ell^{p_2}}
 \gtrsim E_*.
\]
Since each $x^{(j)}$ is supported on a single shell,
\[
 \|x^{(j)}\|_{\ell^{p_1,q_1}}\leq1.
\]
By Proposition~\ref{prop:toeplitz-testing-matrix}, after passing to a
sufficiently sparse subsequence of the shells,
\begin{equation}\label{eq:5-separated-testing}
 \|C^\mu x^{(i)}-C^\mu x^{(j)}\|_{\ell^{p_2,q_2}}
 \gtrsim E_*,
 \qquad i\neq j.
\end{equation}
Here Proposition~\ref{prop:toeplitz-testing-matrix}(c) is applied with a fixed lower
bound $\varepsilon_0=cE_*$, where $c>0$ depends only
on the fixed structural and geometric parameters.
The required spacing of the selected shells may depend
on $\frac{M}{E_*}$, but the implicit constant in \eqref{eq:5-separated-testing}
is independent of $\mu$, the selected shell indices,
and the compact operator $K$ considered below.

\smallskip
\noindent
Set
\[
 B
 :=
 R_{\alpha,\Omega}^{p_2,q_2}
 T_\mu^\alpha
 S_{\alpha,\Omega}^{p_1,q_1}.
\]
For the finitely supported vectors $x^{(j)}$,
\[
 Bx^{(j)}=C^\mu x^{(j)}.
\]
Let
\(
 K:
 \mathcal F_\alpha^{p_1,q_1}
 \longrightarrow
 \mathcal F_\alpha^{p_2,q_2}
\)
be compact and put
\(
 L
 :=
 R_{\alpha,\Omega}^{p_2,q_2}
 K
 S_{\alpha,\Omega}^{p_1,q_1}.
\)
Then $L$ is compact. Hence, after passing to a subsequence,
$(Lx^{(j)})_{j \geq 1}$ is Cauchy in $\ell^{p_2,q_2}(\Omega)$. By
\eqref{eq:5-separated-testing} and the $s_2$-triangle inequality,
\[
 E_*^{\,s_2}
 \lesssim
 2\|B-L\|^{s_2}
 +
 \|Lx^{(i)}-Lx^{(j)}\|_{\ell^{p_2,q_2}}^{s_2}.
\]
Letting $i,j\to\infty$ gives
\[
 \|B-L\|_{\ell^{p_1,q_1}\to\ell^{p_2,q_2}}
 \gtrsim E_*.
\]
Since the restriction and synthesis maps are bounded,
\[
 \|B-L\|
 \lesssim
 \|T_\mu^\alpha-K\|_{\mathcal F_\alpha^{p_1,q_1}
 \to\mathcal F_\alpha^{p_2,q_2}}.
\]
The implicit constants are independent of $K$:
although the Cauchy subsequence extracted from
compactness may depend on $K$, the separation
estimate~\eqref{eq:5-separated-testing} holds uniformly on every subsequence.
Taking the infimum over all compact $K$ therefore yields
\begin{equation}\label{eq:5-essential-lower}
 \|T_\mu^\alpha\|_{\rm ess}
 \gtrsim E_*.
\end{equation}
Combining~\eqref{eq:5-essential-upper} and
\eqref{eq:5-essential-lower}, we obtain
\[
 \|T_\mu^\alpha\|_{\rm ess}\asymp E_*.
\]
Finally, Proposition~\ref{prop:local-mass-radius} and
Theorem~\ref{thm:local-mass-berezin} give, for every fixed $R>0$,
\[
  \|T_\mu^\alpha\|_{\rm ess}
 \asymp
 \limsup_{k\to\infty}
 \big\|
 (1+|\cdot|)^{\nu_2-\nu_1}\widehat\mu_R
 \big\|_{L^{p_{12}}(A_k,dA)}
 \asymp
 \limsup_{k\to\infty}
 \big\|
 (1+|\cdot|)^{\nu_2-\nu_1}\widetilde\mu_\alpha
 \big\|_{L^{p_{12}}(A_k,dA)}.
\]
In particular, $T_\mu^\alpha$ is compact if and only if either of
these shell norms tends to zero. This completes the compactness and
essential-norm assertions of Theorem~\ref{thm:q1leq-q2}.

\subsection{Mappings involving the little outer endpoint}

Restriction and reconstruction from Subsection~\ref{subsec:lattice}
give the characterization
\begin{equation}\label{eq:little-endpoint-characterization}
 f\in f_\alpha^{p,\infty}
 \quad\Longleftrightarrow\quad
 f\in\mathcal F_\alpha^{p,\infty}
 \ \text{and}\
 R_{\alpha,\Omega}^{p,\infty}f
 \in\ell_0^{p,\infty}(\Omega).
\end{equation}
Indeed, the forward implication is the little-endpoint restriction
estimate, while the reverse follows from the canonical dual-atom
reconstruction.

\begin{proof}[Proof of Theorem~1.3]

\emph{(a) Boundedness.} Suppose that
\(
 T_\mu^\alpha:
 f_\alpha^{p_1,\infty}
 \longrightarrow
 \mathcal F_\alpha^{p_2,q_2}
\)
is bounded. Since $q_1=\infty$, one has
\(
 q_{12}=q_2.
\)
Let
\[
 d_\lambda
 :=
 \frac{\alpha}{\pi}
 (1+|\lambda|)^{\nu_2-\nu_1}
 \widetilde\mu_\alpha(\lambda),
 \qquad \lambda\in\Omega.
\]
Fix a finite union $E$ of lattice shells and let $x$ be supported in
$E$. Then
\(S_{\alpha,\Omega}^{p_1,\infty}x\)
is a finite normalized-kernel sum and therefore belongs to
$f_\alpha^{p_1,\infty}$. On such inputs,
\[
 R_{\alpha,\Omega}^{p_2,q_2}
 T_\mu^\alpha
 S_{\alpha,\Omega}^{p_1,\infty}x
 =
 C^\mu x.
\]
By \eqref{eq:testing-diagonal}, the diagonal of $C^\mu$ is
$(d_\lambda)_{\lambda\in\Omega}$. The finite-set Rademacher
extraction from Subsection~\ref{subsec:diagonal}, followed by
Proposition~\ref{prop:annular-multipliers}, gives
\[
 \|d\,\mathbf1_E\|_{\ell^{p_{12},q_2}}
 \lesssim
 \|T_\mu^\alpha\|_{
 f_\alpha^{p_1,\infty}\to\mathcal F_\alpha^{p_2,q_2}}.
\]
The implicit constant is independent of $E$ and $\mu$.
Increasing $E$ through finite unions of shells and using monotonicity
of the shell quasi-norms yields
\(
 d\in\ell^{p_{12},q_2}(\Omega)
\)
with the same bound.

\smallskip
\noindent
By \eqref{eq-local-mass-berezin-pointwise},
the sequence $m = (m_\lambda)_{\lambda \in \Omega}$ defined in
\eqref{eq-weighted-local-mass}
belongs to $\ell^{p_{12},q_2}(\Omega)$.
Proposition~\ref{prop:cell-discretization} and
Theorem~\ref{thm:local-mass-berezin} therefore give the corresponding continuous
local-mass and Berezin criteria. Thus the criterion for the
little-source mapping is necessary.

\smallskip
\noindent
Conversely, if this criterion holds, Theorem~\ref{thm:q1leq-q2} when $q_2=\infty$
and Theorem~~\ref{thm:q2lt-q1} when $q_2<\infty$ give the corresponding bounded
operator
\[
 T_\mu^\alpha:
 \mathcal F_\alpha^{p_1,\infty}
 \longrightarrow
 \mathcal F_\alpha^{p_2,q_2}.
\]
Restriction to $f_\alpha^{p_1,\infty}$ proves sufficiency. The
preceding finite-set argument gives the reverse norm estimate, so
the same norm equivalences as in Theorems~\ref{thm:q1leq-q2}--~\ref{thm:q2lt-q1} remain valid.

\smallskip
\noindent
Assume now that $q_2=\infty$. Under the boundedness criterion,
\(
 m\in\ell^{p_{12},\infty}(\Omega).
\)
Let
\( f\in f_\alpha^{p_1,\infty}\) and \(x=R_{\alpha,\Omega}^{p_1,\infty}f\).
By \eqref{eq:little-endpoint-characterization},
\(
 x\in\ell_0^{p_1,\infty}(\Omega).
\)
Proposition~\ref{prop:endpoint-reconstruction} gives
\[
 R_{\alpha,\Omega}^{p_2,\infty}T_\mu^\alpha f
 =
 D^\mu x,
 \qquad
 |D^\mu x|
 \lesssim
 GM_mK_L|x|.
\]
The shell estimate established in the proof of that proposition
shows that $K_L$ preserves $\ell_0^{p_1,\infty}$. Moreover,
\[
 \|(M_m y)|_{\Omega_k}\|_{\ell^{p_2}}
 \leq
 \|m|_{\Omega_k}\|_{\ell^{p_{12}}}
 \|y|_{\Omega_k}\|_{\ell^{p_1}},
\]
so boundedness of the block profile implies that $M_m$ preserves
the little outer endpoint. Finally,
Proposition~\ref{prop:gaussian-matrices} shows that $G$ preserves
it as well. Hence
\(
 R_{\alpha,\Omega}^{p_2,\infty}T_\mu^\alpha f
 \in\ell_0^{p_2,\infty},
\)
and \eqref{eq:little-endpoint-characterization} gives
\(
 T_\mu^\alpha f\in f_\alpha^{p_2,\infty}.
\)
Thus the same boundedness criterion also characterizes
\[
 T_\mu^\alpha:
 f_\alpha^{p_1,\infty}
 \longrightarrow
 f_\alpha^{p_2,\infty}.
\]

\smallskip
\noindent
\emph{(b) Compactness.}
 If $q_2<\infty$, the criterion
already gives compactness of the corresponding large-source
operator by Theorem~\ref{thm:q2lt-q1}, and hence compactness of its restriction
to the little source.

\smallskip
\noindent
Let $q_2=\infty$.
By Proposition~\ref{prop:cell-discretization} and Theorem~\ref{thm:local-mass-berezin}, the compactness criterion is
equivalent to
\[
 \|m|_{\Omega_k}\|_{\ell^{p_{12}}}\longrightarrow0.
\]
If this holds, Theorem~\ref{thm:q1leq-q2} gives compactness of the corresponding
large-to-large operator, hence also of its restriction to the
little source. For the little target, the preceding boundedness
argument places the range in the closed subspace
$f_\alpha^{p_2,\infty}$; compactness in the inherited quasi-norm
therefore follows.

\smallskip
\noindent
Conversely, suppose that the little-source operator is compact and
that $\|m|_{\Omega_k}\|_{\ell^{p_{12}}}$ does not tend to zero. Since compactness implies
boundedness, the preceding argument gives
\[
 M:= \|m\|_{\ell^{p_{12}, \infty}} = \sup_{k \geq 0}  \|m|_{\Omega_k}\|_{\ell^{p_{12}}}<\infty.
\]
Choose $k_j\to\infty$ and $\varepsilon>0$ such that
\( \|m|_{\Omega_{k_j}}\|_{\ell^{p_{12}}} \geq  \varepsilon.
\)
The pointwise comparison
\eqref{eq-local-mass-berezin-pointwise},
together with \eqref{eq:testing-diagonal}, the block multiplier
identity, and finite-set Rademacher extraction, gives vectors
$x^{(j)}$ supported in $\Omega_{k_j}$ such that
\[
 \|x^{(j)}\|_{\ell^{p_1}}\leq1,
 \qquad
 \|\mathbf1_{\Omega_{k_j}}
 C^\mu x^{(j)}\|_{\ell^{p_2}}
 \geq c\varepsilon.
\]
After passing to a sufficiently sparse subsequence,
Proposition~\ref{prop:toeplitz-testing-matrix} makes the vectors
$C^\mu x^{(j)}$ pairwise separated in
$\ell^{p_2,\infty}(\Omega)$.

\smallskip
\noindent
Set
\[
 f_j
 :=
 S_{\alpha,\Omega}^{p_1,\infty}x^{(j)}.
\]
Each $f_j$ is a finite kernel sum, hence belongs to
$f_\alpha^{p_1,\infty}$, and the sequence $(f_j)_{j \geq 1}$ is bounded there.
Since
\(
 R_{\alpha,\Omega}^{p_2,\infty}T_\mu^\alpha f_j
 =
 C^\mu x^{(j)},
\)
boundedness of the normalized restriction gives
\[
 \|T_\mu^\alpha f_i-T_\mu^\alpha f_j\|_
 {\mathcal F_\alpha^{p_2,\infty}}
 \gtrsim \varepsilon
 \qquad (i\neq j).
\]
Thus the image sequence has no Cauchy subsequence, contradicting
compactness. The same argument applies when the target is the
little space, since it carries the inherited quasi-norm. This proves
the compactness assertions.
\end{proof}

\begin{proof}[Proof of Theorem~1.4\textup{(a)}]
Assume $0<q_1<\infty$ and $q_2=\infty$. Necessity follows immediately
from the continuous inclusion
\(
 f_\alpha^{p_2,\infty}
 \subset
 \mathcal F_\alpha^{p_2,\infty}
\)
and Theorem~\ref{thm:q1leq-q2}.

\smallskip
\noindent
Conversely, suppose that the large-target boundedness criterion
holds. By Proposition~\ref{prop:local-mass-radius} we may work at the radius $R_*$, and
Proposition~\ref{prop:cell-discretization}\textup{(a)}, applied to the square cell system
associated with $\Omega$, gives
\[
m = (m_{\lambda})_{\lambda \in \Omega}\in\ell^{p_{12},\infty}(\Omega),
\]
where $m$ is defined in
\eqref{eq-weighted-local-mass}.
Let
\( f\in\mathcal F_\alpha^{p_1,q_1}\) and \(x=R_{\alpha,\Omega}^{p_1,q_1}f\).
Then
\[
 x\in\ell^{p_1,q_1}(\Omega).
\]
Since $q_1<\infty$,
\(
 \|x|_{\Omega_k}\|_{\ell^{p_1}}\longrightarrow0.
\)
By Proposition~\ref{prop:endpoint-reconstruction}, for $L$
sufficiently large,
\[
 R_{\alpha,\Omega}^{p_2,\infty}T_\mu^\alpha f
 =
 D^\mu x,
 \qquad
 |D^\mu x|
 \lesssim
 GM_mK_L|x|.
\]
The same proposition gives the required boundedness of $K_L$ on
$\ell^{p_1,q_1}(\Omega)$. Hence
\(
 K_L|x|\in\ell^{p_1,q_1}(\Omega),
\)
and, since $q_1<\infty$,
\[
 \|(K_L|x|)|_{\Omega_k}\|_{\ell^{p_1}}
 \longrightarrow0.
\]
Now apply the block estimate
\eqref{eq:block-multiplier} from
Proposition~\ref{prop:annular-multipliers}. Since
$m\in\ell^{p_{12},\infty}(\Omega)$,
\[
 \begin{aligned}
 \|(M_mK_L|x|)|_{\Omega_k}\|_{\ell^{p_2}}
 &\leq
 \|m|_{\Omega_k}\|_{\ell^{p_{12}}}
 \|(K_L|x|)|_{\Omega_k}\|_{\ell^{p_1}}
 \longrightarrow0.
 \end{aligned}
\]
Thus
\[
 M_mK_L|x|\in\ell_0^{p_2,\infty}(\Omega).
\]
By Proposition~\ref{prop:gaussian-matrices}, the Gaussian matrix
$G$ preserves the little outer endpoint, and therefore
\[
 D^\mu x\in\ell_0^{p_2,\infty}(\Omega).
\]
Hence
\( R_{\alpha,\Omega}^{p_2,\infty}T_\mu^\alpha f
 \in\ell_0^{p_2,\infty}(\Omega),
\)
and \eqref{eq:little-endpoint-characterization} yields
\[
 T_\mu^\alpha f\in f_\alpha^{p_2,\infty}.
\]
The norm estimate is the same as for the corresponding large-target
mapping, by Proposition~\ref{prop:endpoint-reconstruction} together
with Propositions~\ref{prop:annular-multipliers} and
\ref{prop:gaussian-matrices}.

\smallskip
\noindent
For compactness, by Proposition~\ref{prop:local-mass-radius} and
Proposition~\ref{prop:cell-discretization}\textup{(b)}, the large-target compactness criterion
is equivalent to the vanishing of the corresponding discrete shell
profile,
\[
 \|m|_{\Omega_k}\|_{\ell^{p_{12}}}\longrightarrow0.
\]
Equivalently, by Theorem~\ref{thm:local-mass-berezin}, one may use the corresponding
Gaussian Berezin tail.

\smallskip
\noindent
If this condition holds, Theorem~\ref{thm:q1leq-q2} gives compactness
into $\mathcal F_\alpha^{p_2,\infty}$, while the preceding boundedness
argument places the whole range in the closed subspace
$f_\alpha^{p_2,\infty}$. Hence the operator is compact into the
little target.

\smallskip
\noindent
Conversely, compactness into $f_\alpha^{p_2,\infty}$ implies
compactness into $\mathcal F_\alpha^{p_2,\infty}$ through the
continuous inclusion
\(
 f_\alpha^{p_2,\infty}
 \hookrightarrow
 \mathcal F_\alpha^{p_2,\infty}.
\)
Theorem~\ref{thm:q1leq-q2}, together with
Propositions~\ref{prop:local-mass-radius} and~\ref{prop:cell-discretization}\textup{(b)}, then forces the same vanishing
shell profile.
\end{proof}

\begin{proof}[Proof of Theorem~1.4\textup{(b)}]
Now let $q_1=q_2=\infty$. Then
\[
 \nu_2-\nu_1
 =
 \frac1{p_1}-\frac1{p_2},
 \qquad
 \frac1{p_{12}}
 =
 \left(\frac1{p_2}-\frac1{p_1}\right)_+.
\]
Let $m = (m_{\lambda})_{\lambda \in \Omega}$ be the weighted local-mass sequence
defined in \eqref{eq-weighted-local-mass}.
By Proposition~\ref{prop:cell-discretization} and Theorem~\ref{thm:local-mass-berezin}, the local-mass and Gaussian
Berezin conditions in the theorem are both equivalent to
\begin{equation}\label{eq:large-little-discrete}
 \|m|_{\Omega_k}\|_{\ell^{p_{12}}}\longrightarrow0.
\end{equation}
Assume first \eqref{eq:large-little-discrete}. Then
$m\in\ell^{p_{12},\infty}(\Omega)$, so the corresponding
large-to-large operator is bounded. Let
\( f\in\mathcal F_\alpha^{p_1,\infty}\) and \(x=R_{\alpha,\Omega}^{p_1,\infty}f\).
By Proposition~\ref{prop:endpoint-reconstruction},
\[
 R_{\alpha,\Omega}^{p_2,\infty}T_\mu^\alpha f
 =
 D^\mu x,
 \qquad
 |D^\mu x|
 \lesssim
 GM_mK_L|x|.
\]
The sequence $K_L|x|$ belongs to $\ell^{p_1,\infty}$. Since
$\|m|_{\Omega_k}\|_{\ell^{p_{12}}}\to0$,
\[
 \|(M_mK_L|x|)|_{\Omega_k}\|_{\ell^{p_2}}
 \leq
 \|m|_{\Omega_k}\|_{\ell^{p_{12}}}
 \|(K_L|x|)|_{\Omega_k}\|_{\ell^{p_1}}
 \longrightarrow0.
\]
Hence
\( M_mK_L|x|\in\ell_0^{p_2,\infty}(\Omega),
\)
and Proposition~\ref{prop:gaussian-matrices} gives
\(
 D^\mu x\in\ell_0^{p_2,\infty}(\Omega).
\)
By \eqref{eq:little-endpoint-characterization},
\[
 T_\mu^\alpha f\in f_\alpha^{p_2,\infty}.
\]
This proves boundedness into the little target, with the same norm
estimate as the corresponding large-target operator.

\smallskip
\noindent
Moreover, \eqref{eq:large-little-discrete} is precisely the
large-to-large compactness criterion of Theorem~\ref{thm:q1leq-q2}. Thus
$T_\mu^\alpha$ is compact into
$\mathcal F_\alpha^{p_2,\infty}$. Since its range lies in the closed
subspace $f_\alpha^{p_2,\infty}$, it is compact into the little
target as well.

\smallskip
\noindent
Conversely, suppose that
\(
 T_\mu^\alpha:
 \mathcal F_\alpha^{p_1,\infty}
 \longrightarrow
 f_\alpha^{p_2,\infty}
\)
is bounded. By inclusion into the large target and Theorem~\ref{thm:q1leq-q2},
\[
 M:= \|m\|_{\ell^{p_{12}, \infty}} = \sup_{k \geq 0} \|m|_{\Omega_k}\|_{\ell^{p_{12}}} <\infty.
\]
Assume, toward a contradiction, that
$\|m|_{\Omega_k}\|_{\ell^{p_{12}}}\not\to0$. Choose $k_j\to\infty$ and $\varepsilon>0$ such
that
\[
 \|m|_{\Omega_{k_j}}\|_{\ell^{p_{12}}}\geq\varepsilon
 \qquad\text{for all }j.
\]
As above, \eqref{eq:testing-diagonal}, comparison \eqref{eq-local-mass-berezin-pointwise},
the block multiplier identity, and finite-set Rademacher extraction
give a constant
\(
 c=c(\alpha,R_*,p_1,p_2)>0
\)
and vectors $x^{(j)}$ supported in $\Omega_{k_j}$ such that
\begin{equation}\label{eq:large-little-tests}
 \|x^{(j)}\|_{\ell^{p_1}}\leq1,
 \qquad
 \|\mathbf1_{\Omega_{k_j}}C^\mu x^{(j)}\|_{\ell^{p_2}}
 \geq c\varepsilon .
\end{equation}
The constant $c$ is independent of $\mu$, $\varepsilon$, $j$, and
the chosen shells.

\smallskip
\noindent
Let
\(
 s:=\min\{1,p_2\}.
\)
Let $C_0>0$ be the constant in
\eqref{eq:testing-shell}; thus
\[
 \|\mathbf1_{\Omega_\ell}C^\mu x\|_{\ell^{p_2}}
 \leq
 C_0 M
 e^{-\frac{\alpha}{16}|\ell-k|^2}
 \|x\|_{\ell^{p_1}},
 \qquad
 \operatorname{supp}x\subset\Omega_k.
\]
Here $C_0=C_0(\alpha,p_1,p_2,\Omega)>0$ depends only on $\alpha$, the fixed exponents, and the
fixed sampling-lattice geometry, and is independent of $\mu$, $M$,
$k$, $\ell$, and $x$.

\smallskip
\noindent
Passing to a sufficiently sparse subsequence of $(k_j)_{j \geq 1}$, we may
therefore arrange that
\begin{equation}\label{eq:summable-leakage}
 (C_0M)^s
 \sup_{j\geq1}
 \sum_{i\neq j}
 e^{-\frac{\alpha s}{16}|k_i-k_j|^2}
 \leq
 \frac12(c\varepsilon)^s .
\end{equation}
Because the supports of the $x^{(j)}$ lie in distinct shells, the
sequence
\[
 x:=\sum_{j \geq 1} x^{(j)}
\]
belongs to $\ell^{p_1,\infty}(\Omega)$ with
$\|x\|_{\ell^{p_1,\infty}}\leq1$. Put
\[
 f:=S_{\alpha,\Omega}^{p_1,\infty}x
 \in\mathcal F_\alpha^{p_1,\infty}.
\]
The Gaussian majorant in
Proposition~\ref{prop:toeplitz-testing-matrix}, applied before
integration as in its proof, shows that the infinite matrix action is
absolutely defined coordinatewise and justifies interchange of the
kernel synthesis series with the Toeplitz integral. Consequently,
\[
 R_{\alpha,\Omega}^{p_2,\infty}T_\mu^\alpha f
 =
 C^\mu x
 =
 \sum_{i \geq 1} C^\mu x^{(i)}.
\]
For each $j$, the $s$-triangle inequality,
\eqref{eq:large-little-tests}, \eqref{eq:testing-shell}, and
\eqref{eq:summable-leakage} give
\[
 \|
 \mathbf1_{\Omega_{k_j}}
 R_{\alpha,\Omega}^{p_2,\infty}T_\mu^\alpha f
 \|_{\ell^{p_2}}^s
 \geq
 \|\mathbf1_{\Omega_{k_j}}
 C^\mu x^{(j)}\|_{\ell^{p_2}}^s
 -
 \sum_{i\neq j}
 \|\mathbf1_{\Omega_{k_j}}
 C^\mu x^{(i)}\|_{\ell^{p_2}}^s
\geq
 \frac12(c\varepsilon)^s .
\]
Thus the shell norms of
$R_{\alpha,\Omega}^{p_2,\infty}T_\mu^\alpha f$ do not tend to zero,
contradicting \eqref{eq:little-endpoint-characterization}. Hence
$\|m|_{\Omega_{k}}\|_{\ell^{p_{12}}}\to0$.

\smallskip
\noindent
We have therefore proved that boundedness into the little target is
equivalent to \eqref{eq:large-little-discrete}, and this condition
also implies compactness. The converse implication
compactness $\Rightarrow$ boundedness is immediate. This completes
the proof.
\end{proof}

\subsection{Particular cases}

We conclude this section by recording several particular cases of the
preceding results. These cases connect the mixed-norm theory with
familiar situations from the classical Fock-space setting, while also
illustrating features that are specific to the mixed-norm framework.
We consider the one-index scale, radial power densities, drifting bump
measures, and positive function symbols. The first and last cases allow
direct comparison with previously studied Toeplitz operators on Fock
spaces, whereas the intermediate cases make the growth thresholds,
compactness criteria, and endpoint bounded-but-noncompact behavior
explicit.

\subsubsection{The one-index scale.}
Let
\(p_1=q_1=p\) and \( p_2=q_2=q\).
Then
\[
 \nu_1=  \nu_2=0, 
\text{and}
\qquad
\frac1{p_{12}}
 =
 \frac1{q_{12}}
 =
 \left(\frac1q-\frac1p\right)_+.
\]
If $q<p$, then $p_{12}=q_{12}<\infty$, and the boundedness
criterion reduces to
\(
 \widehat\mu_R\in L^{p_{12}}(\mathbb C).
\)
In this range boundedness and compactness are equivalent.

\smallskip
\noindent
Since $z\mapsto\widehat\mu_R(z)$ is lower semicontinuous, the
$L^\infty$ conditions may equivalently be stated pointwise.
If $p\leq q$, then
\(
 p_{12}=q_{12}=\infty,
\)
and boundedness is equivalent to
\[
 \sup_{z\in\mathbb C}\widehat\mu_R(z)<\infty.
\]
Moreover, compactness is equivalent to
\[
 \widehat\mu_R(z)\longrightarrow0 \qquad \text{as} \qquad |z| \to\infty,
\]
and the essential norm satisfies
\[
 \|T_\mu^\alpha\|_{\mathrm{ess}}
 \asymp
 \limsup_{|z|\to\infty}
 \widehat\mu_R(z).
\]
For $1\leq p,q\leq\infty$, the boundedness and compactness
criteria recover the classical results in
\cite{HuLv11,Mengestie}. Essential norms of Toeplitz operators
on Fock spaces have also been studied in
\cite{HuLu, IsralowitzEssential}, while related off-diagonal
essential-norm results for large Fock spaces were obtained in
\cite{Arroussi}.

\subsubsection{Radial power densities.}
For $\beta\in\mathbb R$, let
\[
 d\mu_\beta(z)=(1+|z|)^\beta\,dA(z).
\]
For every fixed $R>0$,
\[
 \widehat\mu_{\beta,R}(z)
 \asymp
 (1+|z|)^\beta.
\]
Hence
\[
 \|(1+|z|)^{\nu_2-\nu_1}
   \widehat\mu_{\beta,R}\|_{L^{p_{12}}(A_k)}
 \asymp
 (1+k)^{\beta+\nu_2-\nu_1+ \frac{1}{p_{12}}}.
\]
If $q_{12}<\infty$, boundedness is equivalent to
\[
\beta+\nu_2-\nu_1+\frac1{p_{12}}<-\frac1{q_{12}}.
\]
In this case boundedness and compactness are equivalent. If
$q_{12}=\infty$, boundedness is equivalent to
\[
 \beta+\nu_2-\nu_1+\frac1{p_{12}}\leq0,
\]
while compactness is equivalent to
\[
 \beta+\nu_2-\nu_1+\frac1{p_{12}}<0.
\]
Thus the borderline case $\beta+\nu_2-\nu_1+\frac1{p_{12}}=0$ gives a bounded noncompact
Toeplitz operator, with  $\|T_{\mu_\beta}^\alpha\|_{\mathrm{ess}}\asymp1$.

\smallskip
\noindent
The choice $\beta=0$ gives Lebesgue measure. In this case
\[
 d\mu_0=dA,
 \qquad
 T_{\mu_0}^\alpha=I
\]
by the reproducing formula \cite[Proposition~2.8]{FT}. Hence the
preceding criterion reduces to the embedding criterion for the
same-Gaussian mixed-norm Fock spaces.

\smallskip
\noindent
If $q_2<q_1$, then $q_{12}<\infty$ and
\[
 \nu_2-\nu_1+\frac1{p_{12}}
 =
 \left(\frac1{q_2}-\frac1{q_1}\right)
 +
 \left(\frac1{p_1}-\frac1{p_2}\right)
 +
 \left(\frac1{p_2}-\frac1{p_1}\right)_+ \geq
 \frac1{q_2}-\frac1{q_1}
 =
 \frac1{q_{12}}>0,
\]
so the boundedness condition cannot hold.

\smallskip
\noindent
If $q_1\leq q_2$, then $q_{12}=\infty$, and boundedness is
equivalent to
\[
 \nu_2-\nu_1+\frac1{p_{12}}\leq0.
\]
If $p_1\leq p_2$, then $p_{12}=\infty$, and this is precisely
\(
 \nu_2\leq\nu_1.
\)
If $p_2<p_1$, then
\[
 \nu_2-\nu_1+\frac1{p_{12}}
 =
 \frac1{q_2}-\frac1{q_1}\leq0,
\]
while $\nu_2<\nu_1$ follows automatically. Consequently,
\(
 T_{\mu_0}^\alpha=I:
 \mathcal F_\alpha^{p_1,q_1}
 \longrightarrow
 \mathcal F_\alpha^{p_2,q_2}
\)
is bounded precisely when
\( q_1\leq q_2\) and \( \nu_2\leq\nu_1\),
which is exactly the same-Gaussian embedding criterion
\cite[Theorem~1.2]{FT}. 

\subsubsection{Drifting bump measures.}
Fix $R_0>0$ and let $(z_j)_{j \geq 1}$ be uniformly separated with
$|z_j|\to\infty$. For $a_j\geq0$, set
\[
 d\mu
 =
 \sum_{j \geq 1}  a_j\mathbf1_{B(z_j,R_0)}\,dA,
 \qquad
 b_j
 :=
 (1+|z_j|)^{\nu_2-\nu_1}a_j.
\]
Let
\[
 Z:=\{z_j:j\geq1\},
 \qquad
 Z_k:=Z\cap A_k, \qquad  J_k:=\{j:k\leq |z_j|<k+1\},
\]
and identify a sequence on $Z$ with $(c_j)_{j \geq 1}$ by writing
\[
 \|c\|_{\ell^{p,q}(Z)}
 :=
 \left\|
 \left(
 \|(c_j)_{j\in J_k}\|_{\ell^p}
 \right)_{k\geq0}
 \right\|_{\ell^q}.
\]
Then, for every fixed $R>0$,
\begin{equation}\label{eq:bumps}
 \|(1+|z|)^{\nu_2-\nu_1}\widehat\mu_R\|
 _{L_{\rm sh}^{p_{12},q_{12}}}
 \asymp  \|b\|_{\ell^{p_{12},q_{12}}(Z)}.
\end{equation}
Indeed,
\[
 \widehat\mu_R(z)
 \lesssim
 \sum_{|z-z_j|<R+R_0}a_j.
\]
Uniform separation gives bounded overlap of the enlarged disks, while
\[
 1+|z|\asymp1+|z_j|
 \qquad\text{whenever}\qquad
 |z-z_j|<R+R_0.
\]
This yields the upper estimate in \eqref{eq:bumps}. For the reverse
estimate, choose $r>0$ so small that the disks $B(z_j,r)$ are
disjoint and
\[
 r<\frac14\min\{R,R_0\}.
\]
For $z\in B(z_j,r)$, the intersection
$B(z,R)\cap B(z_j,R_0)$ contains a disk of fixed positive radius.
Hence
\[
 (1+|z|)^{\nu_2-\nu_1}\widehat\mu_R(z)
 \gtrsim b_j
 \qquad
 (z\in B(z_j,r)).
\]
Since these test disks are disjoint and have uniformly positive
area, regrouping over finitely many neighboring shells gives the
reverse estimate in \eqref{eq:bumps}. Since the enlarged and test disks meet only a fixed number of
neighboring shells, the same estimates hold after deleting the first
$N$ shells, up to shifting $N$ by a fixed amount. Hence the
corresponding tail quasi-norms are equivalent as well. The implicit constants in \eqref{eq:bumps} and the corresponding
tail comparisons depend only on the fixed radii $R,R_0$,
the exponent pairs, and the separation constant of $Z$.

\smallskip
\noindent
Consequently,
\(
 T_\mu^\alpha:
 \mathcal F_\alpha^{p_1,q_1}
 \longrightarrow
 \mathcal F_\alpha^{p_2,q_2}
\)
is bounded if and only if
\[
 b\in\ell^{p_{12},q_{12}}(Z),
\qquad \text{or equivalently,} \qquad
 \big((1+|z_j|)^{\nu_2-\nu_1}a_j\big)_{j \geq 1}
 \in\ell^{p_{12},q_{12}}(Z).
\]
If $q_{2}<q_1$, this condition already implies compactness.
If $q_{1} \leq q_2$, boundedness is equivalent to
\( b\in\ell^{p_{12},\infty}(Z),\)
whereas compactness is equivalent to
\( b\in\ell_0^{p_{12},\infty}(Z).\)
Moreover,
\[
 \|T_\mu^\alpha\|_{\mathrm{ess}}
 \asymp
 \limsup_{k\to\infty}
 \|(b_j)_{j\in J_k}\|_{\ell^{p_{12}}}.
\]
In particular, suppose $q_{1} \leq q_2$ and  choose an increasing sequence of shell indices $k_j\to\infty$,
with the shells sufficiently far apart, and choose one center
$z_j\in A_{k_j}$ for each $j$. If
\(
 a_j=(1+|z_j|)^{\nu_1-\nu_2},
\)
then $b_j=1$. Hence
\[
 b\in\ell^{p_{12},\infty}(Z)
 \setminus
 \ell_0^{p_{12},\infty}(Z).
\]
The corresponding Toeplitz operator is therefore bounded but not
compact, and
\[
 \|T_\mu^\alpha\|_{\mathrm{ess}}\asymp1.
\]

\subsubsection{Positive function symbols.}
Let $\varphi\geq0$ be locally integrable and let
\[
 d\mu_\varphi(z)=\varphi(z)\,dA(z).
\]
Whenever the integral is defined, $T_{\mu_\varphi}^\alpha$ is the
usual Toeplitz operator with positive function symbol $\varphi$:
\[
 T_{\mu_\varphi}^\alpha f(z)
 =
 \frac{\alpha}{\pi}
 \int_{\mathbb C}
 f(w)\varphi(w)e^{\alpha z\overline w-\alpha|w|^2}\,dA(w)
 =
 P_\alpha(\varphi f)(z),
\]
where $P_\alpha$ denotes the usual Fock projection.
Moreover,
\[
 \widehat\mu_{\varphi,R}(z)
 =
 \int_{B(z,R)}\varphi(w)\,dA(w).
\]
Hence
\(
 T_{\mu_\varphi}^\alpha:
 \mathcal F_\alpha^{p_1,q_1}
 \longrightarrow
 \mathcal F_\alpha^{p_2,q_2}
\)
is bounded if and only if
\[
 (1+|z|)^{\nu_2-\nu_1}
 \int_{B(z,R)}\varphi(w)\,dA(w)
 \in
 L_{\rm sh}^{p_{12},q_{12}}.
\]
If $q_2<q_1$, boundedness and compactness are equivalent.
If $q_1 \leq q_2$, compactness is equivalent to
\[
 \left\|
 (1+|z|)^{\nu_2-\nu_1}
 \int_{B(z,R)}\varphi(w)\,dA(w)
 \right\|_{L^{p_{12}}(A_k)}
 \longrightarrow0
 \qquad (k\to\infty),
\]
and
\[
 \|T_{\mu_\varphi}^\alpha\|_{\mathrm{ess}}
 \asymp
 \limsup_{k\to\infty}
 \left\|
 (1+|z|)^{\nu_2-\nu_1}
 \int_{B(z,R)}\varphi(w)\,dA(w)
 \right\|_{L^{p_{12}}(A_k)}.
\]
Thus, for positive function symbols, the relevant quantity is not
the pointwise size of $\varphi$, but its fixed-radius local averages.
Toeplitz operators with function symbols on the classical Fock spaces
have been studied extensively; see, for example,
\cite{BergerCoburn87,CoburnLiIsralowitz,GrudskyVasilevski,Stroethoff92,Stroethoff98}.
The preceding characterization extends this classical
function-symbol setting to the mixed-norm Fock spaces, including the
quasi-Banach and outer-endpoint regimes.

\section{Positive Toeplitz sesquilinear forms}
\label{sec:forms}

The positive Toeplitz form requires a separate argument from the
operator results of Section~5, since in the quasi-Banach target range
a form estimate does not in general determine the operator quasi-norm.
Recall the generalized conjugate $t^*$ and the form exponents
$\hat p,\hat q,\sigma_2$ introduced in Subsection~\ref{subsec:forms}.
We shall also use the following finite-support norming identity. If
$0<p,q\leq\infty$ and $c$ is finitely supported on a locally finite
configuration $\Lambda$, then
\begin{equation}\label{eq:finite-norming}
 \sup_{\substack{
       a\in c_{00}(\Lambda)\\
       \|a\|_{\ell^{p,q}}\leq1}}
 \left|
   \sum_{\lambda\in\Lambda}a_\lambda c_\lambda
 \right|
 =
 \|c\|_{\ell^{p^*,q^*}}.
\end{equation}
See \cite[Lemma~7.2]{SI}. We shall apply
\eqref{eq:finite-norming} only to finite truncations, so no
infinite-dimensional duality assertion is involved.

\subsection{A weighted one-Carleson estimate}

\begin{proposition}\label{prop:w-carleson}
Let $\beta>0$, $0<p,q\leq\infty$, and $\sigma\in\mathbb R$, and put
\[
 \rho:=\frac1p-\frac1q-\sigma.
\]
Let $\mu$ be a positive locally finite Borel measure. If, for some
fixed $R>0$,
\[
 (1+|z|)^\rho\widehat\mu_R(z)
 \in L_{\rm sh}^{p^*,q^*}(\mathbb C),
\]
then, for every $h\in\mathcal F_{\beta;\sigma}^{p,q}$,
\[
 \int_{\mathbb C}
 |h(z)|e^{-\frac{\beta}{2}|z|^2}\,d\mu(z)
 \lesssim
 \left\|
 (1+|\cdot|)^\rho\widehat\mu_R
 \right\|_{L_{\rm sh}^{p^*,q^*}}
 \|h\|_{\mathcal F_{\beta;\sigma}^{p,q}}.
\]
In particular, the integral on the left is absolutely convergent. 
The implicit constant depends only on $\beta,p,q,\sigma$,
the fixed radius $R$, and the fixed geometric data.
\end{proposition}

\begin{proof}
By Proposition~\ref{prop:local-mass-radius}, we may replace $R$ by the fixed radius $R_*$.
We use the square-cell system $(Q_\omega)_{\omega\in\Omega}$ from
Subsection~\ref{subsec:lattice}. Set
\[
 d_\omega
 :=
 (1+|\omega|)^{-\rho}
 \sup_{z\in Q_\omega}
 |h(z)|e^{-\frac{\beta}{2}|z|^2}
\qquad \text{and} \qquad
 m_\omega
 :=
 (1+|\omega|)^\rho
 \widehat\mu_{R_*}(\omega).
\]
We first show that
\begin{equation}\label{eq:one-carleson-d}
 \|d\|_{\ell^{p,q}(\Omega)}
 \lesssim
 \|h\|_{\mathcal F_{\beta;\sigma}^{p,q}}.
\end{equation}
Suppose first that $p<\infty$. Since
\(
 Q_\omega\subset B(\omega,\Delta_\Omega),
\)
the Fock-weighted submean estimate
on a fixed disk (\cite[Lemma~2.32]{Zhu}) gives
\[
 \sup_{z\in Q_\omega}
 |h(z)|^p e^{-\frac{p\beta}{2} |z|^2}
 \lesssim
 \int_{B(\omega,2\Delta_\Omega)}
 |h(w)|^p e^{-\frac{p\beta}{2} |w|^2}\,dA(w).
\]
The disks $B(\omega,2\Delta_\Omega)$ have bounded overlap.
If $\omega\in\Omega_k$, they meet only finitely many annuli $A_j$
with $|j-k|\leq J$, where $J$ is fixed. Since the area of $A_k$
is comparable to $1+k$, the definition of $d_\omega$ gives
\[
 \|d|_{\Omega_k}\|_{\ell^p}
 \lesssim
 (1+k)^{\frac{1}{q}}
 \left(
 \sum_{|j-k|\leq J}
 B_{p,\sigma}(h;j)^p
 \right)^{\frac{1}{p}}.
\]
Taking the outer $\ell^q$ quasi-norm, using finite-shell overlap and
the annular discretization from Subsection~\ref{subsec:mixed-norm-tools}, yields
\eqref{eq:one-carleson-d}.

\smallskip
\noindent
If $p=\infty$, the supremum over each cell $Q_\omega$ is controlled
by the corresponding suprema over finitely many neighboring annuli,
and the same argument gives \eqref{eq:one-carleson-d}.

\smallskip
\noindent
By Proposition~\ref{prop:cell-discretization},
\begin{equation}\label{eq:one-carleson-m}
 \|m\|_{\ell^{p^*,q^*}(\Omega)}
 \asymp
 \left\|
 (1+|\cdot|)^\rho\widehat\mu_{R_*}
 \right\|_{L_{\rm sh}^{p^*,q^*}}.
\end{equation}
Let $F\subset\Omega$ be finite. Since $Q_\omega\subset B(\omega,R_*)$,
\[
 \sum_{\omega\in F}
 \int_{Q_\omega}
 |h(z)|e^{-\frac{\beta}{2}|z|^2}\,d\mu(z)
 \leq
 \sum_{\omega\in F} d_\omega m_\omega.
\]
The finite-support norming identity \eqref{eq:finite-norming} therefore gives
\[
 \sum_{\omega\in F}d_\omega m_\omega
 \leq
 \|d\|_{\ell^{p,q}(\Omega)}
 \|m\|_{\ell^{p^*,q^*}(\Omega)}.
\]
Combining \eqref{eq:one-carleson-d} and
\eqref{eq:one-carleson-m}, we obtain
\[
 \sum_{\omega\in F}
 \int_{Q_\omega}
 |h(z)|e^{-\frac{\beta}{2}|z|^2}\,d\mu(z)
 \lesssim
 \left\|
 (1+|\cdot|)^\rho\widehat\mu_{R_*}
 \right\|_{L_{\rm sh}^{p^*,q^*}}
 \|h\|_{\mathcal F_{\beta;\sigma}^{p,q}}.
\]
The implicit constant is independent of the finite
set $F$. Increasing $F$ through finite unions of shells and applying monotone
convergence gives
\[
 \int_{\mathbb C}
 |h(z)|e^{-\frac{\beta}{2}|z|^2}\,d\mu(z)
 \lesssim
 \left\|
 (1+|\cdot|)^\rho\widehat\mu_{R_*}
 \right\|_{L_{\rm sh}^{p^*,q^*}}
 \|h\|_{\mathcal F_{\beta;\sigma}^{p,q}}.
\]
Finally, Proposition~\ref{prop:local-mass-radius} replaces $R_*$ by the original fixed radius
$R$. In particular, the integral on the left is absolutely
convergent.
\end{proof}

\subsection{Proof of Theorem~\ref{thm:form}}

For the proof, define
\[
 \frac1{\check p}
 :=
 \frac1{p_1}+\frac1{p_2^*},
 \qquad
 \frac1{\check q}
 :=
 \frac1{q_1}+\frac1{q_2^*}.
\]
Then, by the definitions of $\hat p$ and $\hat q$,
\begin{equation}\label{eq:form-hat-check}
 \hat p=(\check p)^*,
 \qquad
 \hat q=(\check q)^*.
\end{equation}
Moreover,
\begin{equation}\label{eq:form-weight-identity}
 \frac1{\check p}
 -
 \frac1{\check q}
 -
 \sigma_2
 =
 \nu_2-\nu_1.
\end{equation}

\smallskip
\noindent
\textit{Sufficiency.}
Let
\( f\in\mathcal F_\alpha^{p_1,q_1}\), \( g\in\mathcal F_{\alpha;\sigma_2}^{p_2^*,q_2^*}\).
Generalized H\"older's inequality on each circle gives
\[
 M_{\check p}(fg,r)
 \leq
 M_{p_1}(f,r)M_{p_2^*}(g,r).
\]
Indeed, when all the relevant exponents are finite, the H\"older
exponents are
\(\frac{p_1}{\check p}\), \(\frac{p_2^*}{\check p}\),
whose reciprocals sum to one; the infinite endpoints are interpreted
in the usual way. Applying the same argument in the radial variable
gives
\begin{equation}\label{eq:form-product}
 \|fg\|_{\mathcal F_{2\alpha;\sigma_2}^{\check p,\check q}}
 \lesssim
 \|f\|_{\mathcal F_\alpha^{p_1,q_1}}
 \|g\|_{\mathcal F_{\alpha;\sigma_2}^{p_2^*,q_2^*}}.
\end{equation}
We apply Proposition~\ref{prop:w-carleson} with
\(
 (p,q,\beta,\sigma)
 =
 (\check p,\check q,2\alpha,\sigma_2).
\)
By \eqref{eq:form-weight-identity}, the corresponding exponent
\[
 \rho
 =
 \frac1{\check p}
 -
 \frac1{\check q}
 -
 \sigma_2 = \nu_2-\nu_1,
\]
while
\eqref{eq:form-hat-check} gives
\(p^*=\hat p\), \(q^*=\hat q\).
Hence
\[
 |\mathcal T_\mu^\alpha(f,g)|
 \leq
 \frac{\alpha}{\pi}
 \int_{\mathbb C}
 |f(w)g(w)|e^{-\alpha|w|^2}\,d\mu(w)
\lesssim
 \left\|
 (1+|\cdot|)^{\nu_2-\nu_1}\widehat\mu_R
 \right\|_{L_{\rm sh}^{\hat p,\hat q}}
 \|fg\|_{\mathcal F_{2\alpha;\sigma_2}^{\check p,\check q}}.
\]
Together with \eqref{eq:form-product}, this yields
\[
 |\mathcal T_\mu^\alpha(f,g)|
 \lesssim
 \left\|
 (1+|\cdot|)^{\nu_2-\nu_1}\widehat\mu_R
 \right\|_{L_{\rm sh}^{\hat p,\hat q}}
 \|f\|_{\mathcal F_\alpha^{p_1,q_1}}
 \|g\|_{\mathcal F_{\alpha;\sigma_2}^{p_2^*,q_2^*}}.
\]
Thus
\begin{equation}\label{eq:form-upper}
 \|\mathcal T_\mu^\alpha\|_{\mathrm{form}}
 \lesssim
 \left\|
 (1+|\cdot|)^{\nu_2-\nu_1}\widehat\mu_R
 \right\|_{L_{\rm sh}^{\hat p,\hat q}}.
\end{equation}
Proposition~\ref{prop:w-carleson} also gives the absolute convergence of the defining
integral.

\smallskip
\noindent
\textit{Necessity.}
Assume that $\mathcal T_\mu^\alpha$ is bounded. Let $x$ and $y$ be
finitely supported sequences on $\Omega$, and let
$(\varepsilon_\lambda)_{\lambda\in\Omega}$ be independent Rademacher
signs. Set
\[
 f_\varepsilon
 :=
 \sum_{\lambda\in\Omega}
 \varepsilon_\lambda x_\lambda
 (1+|\lambda|)^{-\nu_1}
 \kappa_{\alpha,\lambda},
\qquad \text{and} \qquad
 g_\varepsilon
 :=
 \sum_{\lambda\in\Omega}
 \varepsilon_\lambda y_\lambda
 (1+|\lambda|)^{-\nu(p_2^*,q_2^*)-\sigma_2}
 \kappa_{\alpha,\lambda}.
\]
Since the sums defining $f_\varepsilon$ and $g_\varepsilon$ are finite,
we may expand the form and take expectation term by term. Using
\(
 \mathbb E(\varepsilon_\lambda\varepsilon_\gamma)
 =
 \delta_{\lambda\gamma},
\)
all off-diagonal terms vanish. Hence
\[
 \mathbb E\,
 \mathcal T_\mu^\alpha(f_\varepsilon,g_\varepsilon)
 =
 \frac{\alpha}{\pi}
 \sum_{\lambda\in\Omega}
 x_\lambda\overline{y_\lambda}\,
 (1+|\lambda|)^{
 -\nu_1-\nu(p_2^*,q_2^*)-\sigma_2}
 \int_{\mathbb C}
 |\kappa_{\alpha,\lambda}(w)|^2
 e^{-\alpha|w|^2}\,d\mu(w).
\]
For the normalized kernels,
\[
 \int_{\mathbb C}
 |\kappa_{\alpha,\lambda}(w)|^2
 e^{-\alpha|w|^2}\,d\mu(w)
 =
 \widetilde\mu_\alpha(\lambda).
\]
Moreover,
\(
 -\nu_1-\nu(p_2^*,q_2^*)-\sigma_2
 =
 \nu_2-\nu_1.
\)
Consequently,
\[
 \mathbb E\,
 \mathcal T_\mu^\alpha(f_\varepsilon,g_\varepsilon)
 =
 \frac{\alpha}{\pi}
 \sum_{\lambda\in\Omega}
 x_\lambda\overline{y_\lambda}\,
 (1+|\lambda|)^{\nu_2-\nu_1}
 \widetilde\mu_\alpha(\lambda).
\]
On the other hand, the boundedness of the form and the
synthesis estimates give, uniformly in the signs and
the finite supports of $x$ and $y$,

\[
 |\mathcal T_\mu^\alpha(f_\varepsilon,g_\varepsilon)|
 \lesssim
 \|\mathcal T_\mu^\alpha\|_{\mathrm{form}}
 \|x\|_{\ell^{p_1,q_1}}
 \|y\|_{\ell^{p_2^*,q_2^*}}.
\]
Thus
\begin{equation}\label{eq:form-discrete-pairing}
 \left|
 \sum_{\lambda\in\Omega}
 x_\lambda\overline{y_\lambda}\,
 (1+|\lambda|)^{\nu_2-\nu_1}
 \widetilde\mu_\alpha(\lambda)
 \right|
 \lesssim
 \|\mathcal T_\mu^\alpha\|_{\mathrm{form}}
 \|x\|_{\ell^{p_1,q_1}}
 \|y\|_{\ell^{p_2^*,q_2^*}}.
\end{equation}
Put
\[
 \widetilde m_\lambda
 :=
 (1+|\lambda|)^{\nu_2-\nu_1}
 \widetilde\mu_\alpha(\lambda).
\]

\smallskip
\noindent
Every finitely supported
$z\in\ell^{\check p,\check q}(\Omega)$ admits a factorization
\( z_\lambda=x_\lambda\overline{y_\lambda}
\)
such that
\begin{equation}\label{eq:form-factorization}
 \|x\|_{\ell^{p_1,q_1}}
 \|y\|_{\ell^{p_2^*,q_2^*}}
 \leq
 \|z\|_{\ell^{\check p,\check q}}.
\end{equation}
Indeed, on each nonzero shell $\Omega_k$, first factor
$z|_{\Omega_k}$ according to
\(
 \frac1{\check p}
 =
 \frac1{p_1}+\frac1{p_2^*},
\)
with the two factors normalized to have
$\ell^{p_1}$- and $\ell^{p_2^*}$-norm one.  The resulting shell
amplitudes
\( c_k:=\|z|_{\Omega_k}\|_{\ell^{\check p}}
\)
are then factorized according to
\(
 \frac1{\check q}
 =
 \frac1{q_1}+\frac1{q_2^*}.
\)
Combining the two factorizations gives
\eqref{eq:form-factorization}.  If one of the exponents is infinite,
the corresponding factor is taken to have unit modulus on the
relevant finite support.

\smallskip
\noindent
It follows from
\eqref{eq:form-discrete-pairing}--\eqref{eq:form-factorization} that
\[
 \left|
 \sum_{\lambda\in\Omega}
 z_\lambda\widetilde m_\lambda
 \right|
 \lesssim
 \|\mathcal T_\mu^\alpha\|_{\mathrm{form}}
 \|z\|_{\ell^{\check p,\check q}}
\]
for every finitely supported $z$.

\smallskip
\noindent
For $N\geq0$, let
\[
 \widetilde m^{(N)}
 :=
 \widetilde m\,
 \mathbf 1_{\bigcup_{k=0}^{N}\Omega_k}.
\]
Since $\widetilde m^{(N)}$ has finite support, the finite-support
norming identity \eqref{eq:finite-norming} gives
\[
 \|\widetilde m^{(N)}\|_
 {\ell^{(\check p)^*,(\check q)^*}}
 \lesssim
 \|\mathcal T_\mu^\alpha\|_{\mathrm{form}}.
\]
Using \eqref{eq:form-hat-check} and then letting $N\to\infty$, we get
\begin{equation}\label{eq:form-discrete-profile}
 \|\widetilde m\|_{\ell^{\hat p,\hat q}(\Omega)}
 \lesssim
 \|\mathcal T_\mu^\alpha\|_{\mathrm{form}}.
\end{equation}
We now return to the weighted local-mass sequence
$m=(m_\lambda)_{\lambda\in\Omega}$ defined in
\eqref{eq-weighted-local-mass}.
By \eqref{eq-local-mass-berezin-pointwise},
applied with $R=R_*$, and \eqref{eq:form-discrete-profile}, we obtain
\[
\|m\|_{\ell^{\hat p,\hat q}(\Omega)}
\lesssim \|T_\mu^\alpha\|_{\mathrm{form}}.
\]
By Proposition~~\ref{prop:cell-discretization}(a) and the radius stability
in Proposition~\ref{prop:local-mass-radius}, for every fixed $R>0$,
\[
\|(1+|\cdot|)^{\nu_2-\nu_1}\widehat\mu_R
\|_{L_{\mathrm{sh}}^{\hat p,\hat q}}
\asymp
\|m\|_{\ell^{\hat p,\hat q}(\Omega)}.
\]
Consequently,
\begin{equation}\label{eq:form-lower}
\|(1+|\cdot|)^{\nu_2-\nu_1}\widehat\mu_R
\|_{L_{\mathrm{sh}}^{\hat p,\hat q}}
\lesssim \|T_\mu^\alpha\|_{\mathrm{form}}.
\end{equation}
Theorem~\ref{thm:local-mass-berezin} gives the equivalent Gaussian
Berezin condition.

\smallskip
\noindent
Combining \eqref{eq:form-upper} and \eqref{eq:form-lower}, we conclude
that
\[
 \|\mathcal T_\mu^\alpha\|_{\mathrm{form}}
 \asymp
 \left\|
 (1+|\cdot|)^{\nu_2-\nu_1}\widehat\mu_R
 \right\|_{L_{\rm sh}^{\hat p,\hat q}}
 \asymp
 \left\|
 (1+|\cdot|)^{\nu_2-\nu_1}\widetilde\mu_\alpha
 \right\|_{L_{\rm sh}^{\hat p,\hat q}}.
\]
This completes the proof.

\section*{Acknowledgements}

X.~F. is
supported by a grant from the National Science and Technology Council, Taiwan
(NSTC 114-2115-M-A49-003-MY3).

\bigskip
\textit{AI Statement.}
The authors used artificial-intelligence tools for language editing,
\LaTeX\ formatting, and limited assistance with local mathematical
reasoning.  The proof strategy and mathematical development are the
authors' own; all mathematical content was independently written and checked by the authors, who take full responsibility for it.


\begin{thebibliography}{99}
\bibitem{Arroussi} H. Arroussi, H. He, J. Li, and C. Tong, \emph{Toeplitz operators between large Fock spaces}, Banach J. Math. Anal. \textbf{16} (2022), Article 32.

\bibitem{BergerCoburn87}
C.~A. Berger and L.~A. Coburn,
\emph{Toeplitz operators on the Segal--Bargmann space},
Trans. Amer. Math. Soc. \textbf{301} (1987), no.~2, 813--829.

\bibitem{BlascoGalbis}
O. Blasco and A. Galbis,
\emph{Boundedness and compactness of Hausdorff operators on Fock spaces},
Trans. Amer. Math. Soc. \textbf{377} (2024), 5165--5196.

\bibitem{CoburnLiIsralowitz}
L.~A.~Coburn, J.~Isralowitz, and B.~Li,
\emph{Toeplitz operators with BMO symbols on the Segal--Bargmann space},
Trans. Amer. Math. Soc. \textbf{363} (2011), 3015--3030.

\bibitem{ConstantinPelaez}
O. Constantin and J. A. Pel\'aez,
\emph{Integral operators, embedding theorems and a Littlewood--Paley formula
on weighted Fock spaces},
J. Geom. Anal. \textbf{26} (2016), no. 2, 1109--1154.

\bibitem{FangTienRandom}
X. Fang and P. T. Tien,
\emph{Two problems on random analytic functions in Fock spaces},
Canad. J. Math. \textbf{75} (2023), no. 4, 1176--1198.
\bibitem{FT} X. Fang and P. T. Tien, \emph{Sharp Gaussian mixed-norm theory for Fock spaces}, preprint arXiv:2609.26821, 2026.

\bibitem{SI} X. Fang and P. T. Tien, \emph{Sampling and interpolation in Gaussian mixed-norm Fock spaces}, preprint arXiv:2609.23572, 2026.

\bibitem{GrudskyVasilevski}
S. Grudsky and N. Vasilevski,
\emph{Toeplitz operators on the Fock space: radial component effects},
Integral Equations Operator Theory \textbf{44} (2002), no.~1, 10--37.

\bibitem{HuLv11} Z. Hu and X. Lv, \emph{Toeplitz operators from one Fock space to another}, Integral Equations Operator Theory \textbf{70} (2011), no. 4, 541--559.
\bibitem{HuLv14} Z. Hu and X. Lv, \emph{Toeplitz operators on Fock spaces $F^p(\varphi)$}, Integral Equations Operator Theory \textbf{80} (2014), no. 1, 33--59.
\bibitem{HuLu} Z. Hu and J. Lu, \emph{Essential norm of Toeplitz operators on the Fock spaces}, Integral Equations Operator Theory \textbf{83} (2015), no. 2, 197--210.
\bibitem{HuLv17} Z. Hu and X. Lv, \emph{Positive Toeplitz operators between different doubling Fock spaces}, Taiwanese J. Math. \textbf{21} (2017), no. 2, 467--487.

\bibitem{LiuHausdorff}
Y. Liu,
\emph{Hausdorff operators on weighted mixed norm Fock spaces},
Bull. Malays. Math. Sci. Soc. \textbf{48} (2025), no. 2, Paper No. 48.

\bibitem{LiuProjection}
Y. Liu,
\emph{Fock projections on mixed norm spaces},
Mediterr. J. Math. \textbf{21} (2024), no. 6, Paper No. 171.

\bibitem{LiuShiWei}
Y. Liu, H. Shi, and S. Wei,
\emph{Boundedness of Fock projections between mixed norm spaces},
J. Math. Anal. Appl. \textbf{565} (2027), no. 1, Paper No. 130921.

\bibitem{IsralowitzEssential}
J. Isralowitz,
\emph{Compactness and essential norm properties of operators on generalized Fock spaces},
J. Operator Theory \textbf{73} (2015), no.~2, 281--314.

\bibitem{IsralowitzZhu} J. Isralowitz and K. Zhu, \emph{Toeplitz operators on the Fock space}, Integral Equations Operator Theory \textbf{66} (2010), no. 4, 593--611.
\bibitem{LuLv} J. Lu and X. Lv, \emph{Toeplitz operators between Fock spaces}, Bull. Aust. Math. Soc. \textbf{92} (2015), no. 2, 316--324.
\bibitem{Lv} X. Lv, \emph{Carleson measures and Toeplitz operators on doubling Fock spaces}, Chinese Ann. Math. Ser. B \textbf{40} (2019), no. 3, 349--362.
\bibitem{Mengestie} T. Mengestie, \emph{On Toeplitz operators between Fock spaces}, Integral Equations Operator Theory \textbf{78} (2014), no. 2, 213--224.
\bibitem{SchusterVarolin} A. P. Schuster and D. Varolin, \emph{Toeplitz operators and Carleson measures on generalized Bargmann--Fock spaces}, Integral Equations Operator Theory \textbf{72} (2012), no. 3, 363--392.

\bibitem{Stroethoff92}
K. Stroethoff,
\emph{Hankel and Toeplitz operators on the Fock space},
Michigan Math. J. \textbf{39} (1992), no.~1, 3--16.

\bibitem{Stroethoff98}
Sangadji and K.~Stroethoff,
\emph{Compact Toeplitz operators on generalized Fock spaces},
Acta Sci. Math. (Szeged) \textbf{64} (1998), no.~3--4, 657--669.

\bibitem{WangXu} E. Wang and J. Xu, \emph{Toeplitz operators between large Fock spaces in several complex variables}, AIMS Math. \textbf{7} (2022), no. 1, 1293--1306.

\bibitem{ZhuBergmanToeplitz}
K. Zhu,
\emph{Positive Toeplitz operators on weighted Bergman spaces of bounded
symmetric domains},
J. Operator Theory \textbf{20} (1988), no. 2, 329--357.

\bibitem{ZhuOperator}
K. Zhu,
\emph{Operator Theory in Function Spaces},
2nd ed., Mathematical Surveys and Monographs, vol. 138,
American Mathematical Society, Providence, RI, 2007.

\bibitem{Zhu} K. Zhu, \emph{Analysis on Fock Spaces}, Graduate Texts in Mathematics, vol. 263, Springer, New York, 2012.












\end{thebibliography}
\end{document}